\documentclass[final,12pt]{article}
\usepackage[UTF8,scheme = plain]{ctex}
\usepackage[top=1.1in, bottom=1.2in, left=0.8in, right=0.8in]{geometry}
\usepackage{amsmath,amssymb,amsfonts,amsthm,mathrsfs,bm,dsfont}
\usepackage{comment}
\usepackage{graphicx,epstopdf,psfrag}
\usepackage{tabularx,colortbl,dcolumn,booktabs}
\usepackage{hyperref}
\usepackage{cite}
\usepackage{url}
\allowdisplaybreaks[4]
\numberwithin{equation}{section}

\usepackage[numbers,sort&compress]{natbib}
\newcommand{\R}{\mathbb{R}}
\newcommand{\N}{\mathbb{N}}

\newcommand{\E}{\mathbb{E}}
\newcommand{\F}{\mathcal{F}}
\renewcommand{\P}{\mathbb{P}}
\newcommand{\dif}[1]{\mathrm{d}#1}
\newcommand{\diff}[1]{\, \mathrm{d} #1}

\newtheorem{Theorem}{Theorem}[section]
\newtheorem{Definition}[Theorem]{Definition}
\newtheorem{Lemma}[Theorem]{Lemma}
\newtheorem{Proposition}[Theorem]{Proposition}

\newtheorem{Remark}[Theorem]{Remark}

\newtheorem{Assumption}{Assumption}[section]

\begin{document}
\title{Large deviations principles of sample paths and invariant measures of numerical methods for parabolic SPDEs\footnotemark[1]
}
\footnotetext{\footnotemark[1]This work is funded by National Natural Science Foundation of China (No. 11971470, No. 11871068, No. 12031020, No. 12022118, No. 12026428) and by Youth Innovation Promotion Association CAS.
}

\author{Chuchu Chen, Ziheng Chen,
Jialin Hong, Diancong Jin\footnotemark[2]
        \\{\small LSEC, ICMSEC, Academy of Mathematics and Systems Science,}
        \\{\small Chinese Academy of Sciences, Beijing 100190, China;}
        \\{\small School of Mathematical Sciences, University of
                  Chinese Academy of Sciences, Beijing 100049, China}
       }
\footnotetext{\footnotemark[2]Emails:
             chenchuchu@lsec.cc.ac.cn,
             zihengchen@lsec.cc.ac.cn(corresponding author),
             hjl@lsec.cc.ac.cn,
             diancongjin@lsec.cc.ac.cn.
             }
\date{}
\maketitle

\begin{abstract}
	 For parabolic stochastic partial differential equations (SPDEs), we show that the numerical methods, including the spatial spectral Galerkin method and further the full discretization via the temporal accelerated exponential Euler method, satisfy the uniform sample path large deviations. Combining the exponential tail estimate of invariant measures, we establish the large deviations principles (LDPs) of invariant measures of these numerical methods.
	  Based on the error estimate between the rate function of the considered numerical methods and that of the original equation, we prove that these numerical methods can weakly asymptotically preserve the LDPs of sample paths and invariant measures of the original equation. This work provides an approach to proving the weakly asymptotical preservation for the above two LDPs for SPDEs with small noise via  numerical methods,  by means of the minimization sequences.	
\end{abstract}



\section
{Introduction}\label{sec:intro}

The sample path large deviations characterize  the asymptotic behaviour of the probability that stochastic differential equations (SDEs) with small noise deviate from their deterministic counterparts (see e.g., \cite{freidlin1988random,sowers1992large,peszat1994large,cerrai2004large}). On the basis of the uniform sample path large deviations, the LDP of invariant measures of some SDEs with small noise can be established (see e.g., \cite{sowers1992largeinvariant,cerrai2005largeinvariant,
	gadat2013large,brzezniak2017large}), which characterizes  the fluctuation of invariant measures with respect to the noise intensity  on an exponential scale. For a numerical method of an SPDE with small noise, we are interested in whether it can asymptotically preserve these exponential decay rate of the probabilities of rare events in the large deviation estimates for the sample paths and invariant measures of the underlying SPDEs. The investigation on the numerically asymptotical preservation of LDPs
is at an early stage and just a very limited number of papers \cite{chen2021asymptotically,chen2020large,hong2020numerically,
chen2020symplectic,hong2021numerical} are devoted to this topic.
As far as we know, there is few work on the asymptotical preservation of the LDPs of sample paths and invariant measures via numerical methods for SPDEs.

In this work, we focus on the  numerically asymptotical preservation for the LDPs of sample paths and invariant measures  of
the following parabolic SPDEs with small noise on a real separable Hilbert space $H := L^{2}((0,1);\R)$,
\begin{eqnarray}\label{eq:SPDE}
      \diff{X_{x}^{\varepsilon}(t)}
      =
      \big(AX_{x}^{\varepsilon}(t)
      +
      F(X_{x}^{\varepsilon}(t))\big)\diff{t}
      +
      \varepsilon Q^{\frac12}\diff{W(t)},
      \quad t > 0,
      \quad X_{x}^{\varepsilon}(0) = x \in H,
\end{eqnarray}
where $\varepsilon > 0$ denotes the scale of the noise,
$\{W(t)\}_{t \geq 0}$ is an $H$-valued cylindrical Wiener process on a normal filtered probability space $(\Omega,\F,\P; \{\F_t\}_{t \geq 0})$, and $Q$ is a self-adjoint, positive definite bounded linear operator on $H$. Equation \eqref{eq:SPDE} models a variety of random phenomena (see, e.g., \cite{da2014stochastic,lord2014introduction}
and references therein), such as heat conduction in media, turbulent flow in fluid dynamics, molecular collisions in gases and liquids, as well as electric fluctuations in resistors.

The relevant work on LDPs of
invariant measures of numerical methods is only \cite{hong2020numerically}, which is for the case of stochastic ordinary differential equations.
Authors therein show that the midpoint scheme can asymptotically preserve the LDP of invariant measures of Langevin equation with quadratic potential in the small noise limit.
Since the general nonlinear SPDEs suffer from a lack of explicit expressions of invariant measures, the corresponding LDPs of invariant measures of numerical methods usually resort to the uniform sample path large deviations.
The main contributions of this work are:
(i) we give the LDPs of invariant measures of numerical methods for parabolic SPDEs based on the uniform LDP of the corresponding sample paths;
(ii) we provide a framework to establish the weakly asymptotical preservation of the LDPs of sample paths and invariant measures of numerical methods for parabolic SPDEs.

We first discretize \eqref{eq:SPDE} in space by the spectral Galerkin method
\begin{equation}\label{eq:SPDEsemiintro}
      \diff{X_{y}^{\varepsilon,n}(t)}
      =
      \big(P_{n}AX_{y}^{\varepsilon,n}(t)
      +
      P_{n}F(X_{y}^{\varepsilon,n}(t))\big)\diff{t}
      +
      \varepsilon P_{n}Q^{\frac12}\diff{W(t)},
      \quad t > 0,
      \quad X_{y}^{\varepsilon,n}(0) = y := P_{n}x,
\end{equation}
where $P_{n}$ is a projection operator from $H$ to the finite dimensional space $H_{n} \subset H$ (see Subsection \ref{sec:SGM}).
%
%
Following the approach in \cite{peszat1994large}, we obtain
that the family of sample paths of \eqref{eq:SPDEsemiintro} satisfies a uniform LDP with the good rate function $I_{0,T}^{n,y}$ given by \eqref{eq:SPDEsemiratefun} on $C([0,T];H_{n})$; see Theorem \ref{thm:semisolutionLDP}.
Concerning the error analysis of rate functions $I_{0,T}^{n,y}$ and $I_{0,T}^{x}$ (see \eqref{eq:I0TxzI0Txz}), we adopt the definition of weakly asymptotical preservation for LDP (see Definition \ref{def:semiweakasympres})
to deal with the problem caused by different effective domains of rate functions. Noticing the fact that the rate function $I_{0,T}^{x}$ is given by a minimization problem, we utilize the $\dot{H}^{2}$ spatial regularity of skeleton equation \eqref{eq:skeletoneqofSPDE} to simplify its expression. This is an important step in our proof of weakly asymptotical preservation for the LDP of sample paths of the original equations; see Theorem \ref{thm:spatialsolutionwap}.
Based on the uniform LDP of sample paths of \eqref{eq:SPDEsemiintro},
we show that the family of invariant measures
$\{\mu^{\varepsilon,n}\}_{\varepsilon > 0}$ of \eqref{eq:SPDEsemiintro} also satisfies an LDP with the good rate function $V^{n}$ on $H_{n}$.
By approximating both rate functions through some minimization sequences,
and using the error estimate between $I_{0,T}^{n,y}$ and $I_{0,T}^{x}$,
we prove that \eqref{eq:SPDEsemiintro} can weakly asymptotically preserve the LDP of invariant measures for the original equation \eqref{eq:SPDE} in the linear case; see Theorem \ref{th:semiinvameaLDPasympre}.

Let $\tau > 0$ be a uniform time stepsize and $t_{m} := m\tau, m \in \N$ the grid points. We further discretize \eqref{eq:SPDEsemiintro} in time by the accelerated exponential Euler method
\begin{equation}\label{eq:EE}
      Y_{y,m+1}^{\varepsilon,n}
      =
      E_{n}(\tau)Y_{y,m}^{\varepsilon,n}
      +
      \int_{t_{m}}^{t_{m+1}}  E_{n}(t_{m+1}-s)F_{n}(Y_{y,m}^{\varepsilon,n}) \diff{s}
      +
      \varepsilon \int_{t_{m}}^{t_{m+1}}  E_{n}(t_{m+1}-s) Q_{n}^{\frac12}\diff{W_{n}(s)},
\end{equation}
where $Y_{y,0}^{\varepsilon,n} = y$, and $\{E_{n}(t) := e^{tA_{n}}\}_{t \geq 0}$ is a $C_{0}$-semigroup generated by $A_{n}:=P_{n}A$.
To establish the sample path large deviation,
we define $\{Y_{y}^{\varepsilon,n}(t)\}_{t \geq 0}$ as the continuous time approximation of $\{Y_{y,m}^{\varepsilon,n}\}_{m \in \N}$ by
\begin{equation}\label{eq:saptiotemporalconteqintro}
\begin{split}
      &\diff{Y_{y}^{\varepsilon,n}(t)}
      =
      \big(A_{n}Y_{y}^{\varepsilon,n}(t)
      +
      F_{n}(Y_{y}^{\varepsilon,n}(\tau \lfloor t/\tau \rfloor)) \big)\diff{t}
      +
      \varepsilon Q_{n}^{\frac12}\diff{W_{n}(t)},
      \quad t > 0,
      \quad Y_{y}^{\varepsilon,n}(0) = y.
\end{split}
\end{equation}
Following the procedures in Theorems \ref{thm:semisolutionLDP}
and \ref{thm:spatialsolutionwap}, we show that $\{Y_{y}^{\varepsilon,n}\}_{\varepsilon > 0}$
satisfies a uniform LDP on $C([0,T];H_{n})$
in Theorem \ref{th:fullLDPsamplepath}, and weakly asymptotically preserves the sample path large deviation of \eqref{eq:SPDE} in Theorem \ref{th:fullLDPasympre}, respectively. Further, the uniform boundedness of the solution for \eqref{eq:saptiotemporalconteqintro}
implies the existence of an invariant measure $\mu^{\varepsilon,n,\tau}$
for each $\varepsilon > 0$, based on the Krylov--Bogoliubov theorem.
Utilizing the uniform LDP of sample paths of \eqref{eq:saptiotemporalconteqintro},
we show that $\{\mu^{\varepsilon,n,\tau}\}_{\varepsilon > 0}$ satisfies an LDP, and weakly asymptotically preserves the LDP of invariant measures for the original equation \eqref{eq:SPDE} in the linear case; see Theorem \ref{fullinvameaLDPasypre}.

The rest of this paper is organized as follows. In the next section, we  give some assumptions on the considered SPDEs and preliminaries on the theory of large deviations. In Section \ref{sec:spatial}, we apply the spectral Galerkin method to \eqref{eq:SPDE} and establish its LDPs of sample paths and invariant measures. Then the weakly asymptotical preservations for the LDPs of the original SPDEs are given. Section \ref{sec:temporal} focuses on the proof of the LDPs and weakly asymptotical preservations of the fully discrete method.

\section{Preliminaries}\label{sec:preli}

We begin with some notations. Let $C$ be a generic constant that may vary from one place to another, and set $H := L^{2}((0,1);\R)$, equipped with usual product $\langle \cdot,\cdot \rangle$ and norm $|\cdot|$. We denote by $(\mathcal{L}(H),|\cdot|_{\mathcal{L}(H)})$ the Banach space of all bounded linear operators from $H$ to $H$. Besides, let $(\mathcal{L}_{2}(H),\langle\cdot,
\cdot\rangle_{\mathcal{L}_{2}(H)},|\cdot|_{\mathcal{L}_{2}(H)})$ be the real separable Hilbert space  of all Hilbert--Schmidt operators from $H$ to $H$. Throughout this paper, let  $A \colon \mathcal{D}(A) \subset H \to H$ be the Laplacian operator with homogeneous Dirichlet boundary conditions, i.e.,  $Au = \Delta u, u \in \mathcal{D}(A) = H_{0}^{1} \cap H^{2}$.  Then $A$ generates a $C_{0}$-semigroup $\{E(t) := e^{tA}\}_{t \geq 0}$ on $H$ and $-Ae_i=\lambda_{i}e_i, i \in \N^{+}$ with $\{\lambda_{i} := \pi^2i^2\}_{i \in \N^{+}}$ and $\{e_{i}(\xi) := \sqrt{2}\sin(i \pi \xi),\xi \in (0,1)\}_{i \in \N^{+}}$ being an orthonormal basis of $H$. Further, one can define the fractional powers $(-A)^{\gamma}$ of $-A$ and the Hilbert space
     $
         \dot{H}^{\gamma}
     :=
         \mathcal{D}((-A)^{\frac{\gamma}{2}})
     $
for every $\gamma \in \R$, equipped with the inner product
     $
         \langle u,v \rangle_{\dot{H}^{\gamma}}
     =
         \langle
             (-A)^{\frac{\gamma}{2}}u
             ,
             (-A)^{\frac{\gamma}{2}}v
         \rangle
     =
         \sum_{i=1}^{\infty}
         \lambda_{i}^{\gamma}
         \langle u,e_{i} \rangle
         \langle v,e_{i} \rangle
     $
     and the corresponding norm
     $
         |u|_{\dot{H}^{\gamma}}
     =
         \sqrt{\langle u,u \rangle_{\dot{H}^{\gamma}}},
     u,v \in \dot{H}^{\gamma}$; see \cite[Appendix~B.2]{kruse2014strong}.

Assume that the  cylindrical Wiener process $\{W(t)\}_{t \geq 0}$ admits a formal series
$$
W(t) = \sum_{i=1}^{\infty}\beta_{i}(t) e_{i},
\quad t \geq 0,
$$
where $\{\beta_{i}(t)\}_{i \in \N^{+}}, t \geq 0$ is a sequence of independent real-valued standard Brownian motions.
Moreover, assume that $Q \colon H \to H$ is a self-adjoint, positive definite bounded linear operator, and satisfies $Qe_i=q_ie_i$, $i \in \N^{+}$ for a sequence of positive numbers $\{q_i\}_{i \in \N^{+}}$. Denote by $Q^{-\frac{1}{2}} \colon Q^{\frac{1}{2}}(H) \to H$ the inverse of $Q^{\frac{1}{2}}$. To proceed, we make the following assumptions.

\begin{Assumption}\label{ass:AQ}
      Assume that  $\mathcal{D}(A) \subset \mathcal{D}(Q^{-\frac{1}{2}})$ and
      \begin{equation}\label{eq:hybridAQ}
            |(-A)^{\frac{1}{2}}Q^{\frac{1}{2}}|_{\mathcal{L}_{2}(H)}
            <
            \infty,
            \quad
            |Q^{-\frac{1}{2}}(-A)^{-1}|_{\mathcal{L}(H)}
            <
            \infty.
      \end{equation}
\end{Assumption}

\begin{Assumption}\label{ass:F}
      Let $F \colon H \to H$ be Fr\'{e}chet differentiable and there exist $L_{F}, L > 0$ such that
      \begin{align}
      \label{eq:globalLF}
            |F(u)-F(v)|
            \leq&~
            L_{F}|u-v|,
            \quad u,v \in H,
      \\
      \label{eq:auxiliaryassumptionFuH1}
            |F(u)|_{\dot{H}^{1}}
            \leq&~
            L(1+|u|_{\dot{H}^{1}}),
            \quad u \in \dot{H}^{1},
      \\
      \label{eq:auxiliaryassumption}
            |F'(u)v|_{\dot{H}^{2}}
            \leq&~
            L\big(1+|u|_{\dot{H}^{2}}^{2}\big)|v|_{\dot{H}^{2}},
            \quad u,v \in \dot{H}^{2}.
      \end{align}
\end{Assumption}

      Concerning Assumptions \ref{ass:AQ} and \ref{ass:F}, for example,  one may take $q_{i} = \lambda_{i}^{-\delta},i \in \N^{+}$
      for some $\delta \in (\frac{3}{2},2]$
      and $F \colon H \to H$ a Nemytskij operator defined by
      $
            F(u)(\xi) :=  f(u(\xi)),\xi \in (0,1), u \in H     $
      with $f \in C_{b}^{3}(\R)$.

From \cite[Theorem 7.2]{da2014stochastic}, we know that \eqref{eq:SPDE} admits a unique mild solution.
\begin{Proposition}
      Suppose that Assumptions \ref{ass:AQ} and \ref{ass:F} hold. Then \eqref{eq:SPDE} admits a unique mild solution $\{X_{x}^{\varepsilon}(t)\}_{t \geq 0}$ given by
      \begin{equation}\label{eq:SPDEmildsolution}
        X_{x}^{\varepsilon}(t)
        =
        E(t)x
        +
        \int_{0}^{t} E(t-s) F(X_{x}^{\varepsilon}(s)) \diff{s}
        +
        \varepsilon \int_{0}^{t} E(t-s)Q^{\frac12} \diff{W(s)},
        \quad t \geq 0,\quad\P\text{-a.s.}
        \end{equation}
      Moreover, it belongs to $L^{p}(\Omega;C([0,T];H))$ for any $p \geq 1$ and $T > 0$.
\end{Proposition}

To formulate the LDPs for the solution of \eqref{eq:SPDE}, we introduce some preliminaries upon the theory of large deviations; see \cite[Chapter 1]{dembo2009large}. In what follows, let $(\mathcal{U},\rho^{\mathcal{U}})$ be a Polish space and $\mathcal{B}(\mathcal{U})$ its Borel $\sigma$-field.

\begin{Definition}
      A rate function $I$ is a lower semi-continuous mapping $I \colon \mathcal{U} \to [0,+\infty]$ (such that for all $\alpha \geq 0$, the level set $K_{I}(\alpha) := \{u \in \mathcal{U} : I(u) \leq \alpha\}$ is a closed subset of $\mathcal{U}$). A good rate function is a rate function for which all the level sets $K_{I}(\alpha)$ are compact subsets of $\mathcal{U}$. 
\end{Definition}

\begin{Definition}
\label{def:generaldefLDP}
      Let $\{\mu^{\varepsilon}\}_{\varepsilon > 0}$ be a family of probability measures on $(\mathcal{U},\mathcal{B}(\mathcal{U}))$. We say that $\{\mu^{\varepsilon}\}_{\varepsilon > 0}$ satisfies an LDP on $\mathcal{U}$ with the rate $\frac{1}{\varepsilon}$ and the rate function $I$ if
      \begin{equation}\label{eq:LDPclassical}
            -\inf_{u \in U^{o}} I(u)
            \leq
            \liminf_{\varepsilon \to 0} \varepsilon \log \mu^{\varepsilon}(U)
            \leq
            \limsup_{\varepsilon \to 0} \varepsilon \log \mu^{\varepsilon}(U)
            \leq
            -\inf_{u \in \overline{U}} I(u),
            \quad U \in \mathcal{B}(\mathcal{U}),
      \end{equation}
      where $U^{o}$ denotes the interior of $U$,
      and $\overline{U}$ the closure of $U$.
\end{Definition}

The following proposition about Freidlin--Wentzell exponential estimates gives an equivalent characterization of LDP; see \cite[Chapter 3]{freidlin1984random}.

\begin{Proposition}
      If $(\mathcal{U},\rho^{\mathcal{U}})$ is a Polish space and $\{\mu^{\varepsilon}\}_{\varepsilon > 0}$ is a family of probability measures on $(\mathcal{U},\mathcal{B}(\mathcal{U}))$, then $\{\mu^{\varepsilon}\}_{\varepsilon > 0}$ satisfying an LDP on $\mathcal{U}$ with the rate $\frac{1}{\varepsilon}$ and the good rate function $I$ is equivalent to the following statements:
      \begin{enumerate}
        \item [(i)] compact level set: the level set $K_{I}(\alpha) :=  \{u \in \mathcal{U} : I(u) \leq \alpha\}$ is compact for every $\alpha \geq 0$;

        \item [(ii)] lower bound: for any $u \in \mathcal{U}$, $\delta > 0$ and $\gamma > 0$, there exists $\varepsilon_{0} > 0$ such that
            $$\mu^{\varepsilon}(\{v \in \mathcal{U}:\rho^{\mathcal{U}}(u,v) < \delta\})
            \geq \exp\Big(-\frac{I(u)+\gamma}{\varepsilon}\Big),
            \quad \varepsilon \leq \varepsilon_{0};$$

        \item [(iii)] upper bound: for any $\alpha > 0$, $\delta > 0$ and $\gamma > 0$, there exists $\varepsilon_{0} > 0$ such that
            $$\mu^{\varepsilon}(\{v \in \mathcal{U}:
              \rho^{\mathcal{U}}(v,K_{I}(\alpha)) \geq \delta\})
              \leq \exp\Big(-\frac{\alpha-\gamma}{\varepsilon}\Big),
              \quad \varepsilon \leq \varepsilon_{0}.$$
      \end{enumerate}
\end{Proposition}

Regarding $X_{x}^{\varepsilon}$ as a $C([0,T];H)$-value random variable for each $\varepsilon > 0$, the uniform LDP of sample paths of $\{X_{x}^{\varepsilon}\}_{\varepsilon > 0}$ can be summarized as follows; see \cite{peszat1994large} for more details.
\begin{Theorem}
\label{th:exactsoluLDP}
      Suppose that Assumptions \ref{ass:AQ} and \ref{ass:F} hold. Then $\{X_{x}^{\varepsilon}\}_{\varepsilon > 0}$ satisfies a uniform LDP on $C([0,T];H)$ with the rate $\frac{1}{\varepsilon^{2}}$ and the good rate function $I_{0,T}^{x}$ defined by
      \begin{equation}\label{eq:I0TxzI0Txz}
            I_{0,T}^{x}(z)
            =
            \frac{1}{2}
            \inf\{|\varphi|_{L^2(0,T;H)}^{2}:
            \varphi \in L^2(0,T;H),
            z_{0,x}^{\varphi} = z\},
            \quad z \in C([0,T];H),
      \end{equation}
      where $z_{0,x}^{\varphi}$ is the mild solution of the skeleton equation 
      \begin{equation}\label{eq:skeletoneqofSPDE}
            \frac{\dif{z_{0,x}^{\varphi}(t)}}{\dif{t}}
            =
            Az_{0,x}^{\varphi}(t)
            + F(z_{0,x}^{\varphi}(t))
            + Q^{\frac12}\varphi(t),
            \quad t \in (0,T],
            \quad z_{0,x}^{\varphi}(0) = x.
      \end{equation}
\end{Theorem}

To ensure the existence of invariant measures for \eqref{eq:SPDE}, we need the following assumption on dissipativity.

\begin{Assumption}
\label{ass:LFleqlambda1}
      Assume $F(0) = 0$ and $L_{F} < \lambda_{1}$, where $L_{F}$ is the Lipschitz constant of $F$ given by \eqref{eq:globalLF} and $\lambda_{1}$ is the smallest eigenvalue of $-A$.
\end{Assumption}

This assumption together with
$-Ae_{i} = \lambda_{i}e_{i}, i \in \N^{+}$
shows
\begin{equation}\label{eq:dissipativecond}
      \langle Au + F(u), u \rangle
      \leq -c|u|^{2},
      \quad u \in \dot{H}^{2}
\end{equation}
with $c:= \lambda_{1}-L_{F} > 0$.
It follows from \cite{cerrai2005largeinvariant} that there exists $\{t_{i}\}_{i \in \N^{+}} \uparrow +\infty$ (possibly depending on $\varepsilon$) such that the sequence of probability measures $\{\mu_{t_{i}}^{\varepsilon}\}_{i \in \N^{+}}$, defined by
      $$
            \mu_{t_{i}}^{\varepsilon}(B)
            :=
            \frac{1}{t_{i}} \int_{0}^{t_{i}}
            \P\big(X_{0}^{\varepsilon}(s) \in B\big)
            \diff{s},
            \quad B \in \mathcal{B}(H),
      $$
converges weakly to some probability measure $\mu^{\varepsilon}$ on $(H,\mathcal{B}(H))$, which is invariant for \eqref{eq:SPDE}. Moreover, the following theorem shows that $\{\mu^{\varepsilon}\}_{\varepsilon > 0}$ obeys an LDP on $H$; see, e.g., \cite{cerrai2005largeinvariant}.

\begin{Theorem}
      Suppose that Assumptions \ref{ass:AQ}, \ref{ass:F} and \ref{ass:LFleqlambda1} hold. Then  $\{\mu^{\varepsilon}\}_{\varepsilon > 0}$ satisfies an LDP on $H$ with the rate $\frac{1}{\varepsilon^{2}}$ and the good rate function $V$ defined by
      \begin{equation}\label{eq:invaLdpratefun}
            V(u)
            =
            \inf\{I_{0,T}^{0}(z):
            T>0, z \in C([0,T];H), z(0) = 0, z(T) = u\},
            \quad u \in H.
      \end{equation}
\end{Theorem}

\section{Spatial discretization and its LDPs}
\label{sec:spatial}

Here we discretize \eqref{eq:SPDE} in space by the spectral Galerkin method and present its uniform LDP of sample paths in Subsection \ref{sec:SGM}.
Subsection \ref{subsec:semiweakasympre} obtains the weakly asymptotical preservation for the LDP of sample paths of the original equation. After establishing the LDP of invariant measures of the spatial discretization in Subsection \ref{subsec:semiinvameaLDPs}, we prove its weakly asymptotical preservation for the LDP of invariant measures of the original equation in Subsection \ref{sebsec:semiinvameaLDPasy}.

\subsection{Spectral Galerkin method}\label{sec:SGM}

For each $n \in \N^{+}$, we define $H_{n} := \text{span}\{e_{1},\ldots,e_{n}\} \subset H$,  the projection operator $P_{n} \colon H \to H_{n}$ by $P_{n}u = \sum_{i=1}^{n}\langle e_{i},u \rangle e_{i}$ for every $u \in H$, and
$$
      y :=  P_{n}x,
      \quad F_{n} :=  P_{n}F,
      \quad Q_{n}^{\frac12} :=  P_{n}Q^{\frac12},
      \quad W_{n} :=  P_{n}W.
$$
The spectral Galerkin method applied to \eqref{eq:SPDE} is given by
\begin{equation}\label{eq:SPDEsemi}
      \diff{X_{y}^{\varepsilon,n}(t)}
      =
      \big(A_{n}X_{y}^{\varepsilon,n}(t)
      +
      F_{n}(X_{y}^{\varepsilon,n}(t))\big)\diff{t}
      +
      \varepsilon Q_{n}^{\frac12}\diff{W_{n}(t)},
      \quad t > 0,
      \quad X_{y}^{\varepsilon,n}(0) = y,
\end{equation}
where $A_{n} \colon H_{n} \to H_{n}$ is defined by $A_{n} :=  P_{n}A$ and generates a $C_{0}$-semigroup $\{E_{n}(t) := e^{tA_{n}}\}_{t \geq 0}$ on $H_{n}$.
Under Assumptions \ref{ass:AQ} and \ref{ass:F}, Theorem 4.5.3 in \cite{kloeden1992numerical} ensures that \eqref{eq:SPDEsemi} admits a unique mild solution
\begin{equation}\label{eq:SPDEsemisolu}
      X_{y}^{\varepsilon,n}(t)
      =
      E_{n}(t)y
      +
      \int_{0}^{t} E_{n}(t-s) F_{n}(X_{y}^{\varepsilon,n}(s)) \diff{s}
      +
      \varepsilon
      \int_{0}^{t} E_{n}(t-s)Q_{n}^{\frac12} \diff{W_{n}(s)},
      \quad t \geq 0.
\end{equation}
Moreover, the solution process $\{X_{y}^{\varepsilon, n}(t)\}_{t \in [0,T]}$ belongs to $L^{p}(\Omega;C([0,T];H_{n}))$ for any $p \geq 1$ and $T > 0$. To show $\{X_{y}^{\varepsilon,n}\}_{\varepsilon > 0}$ satisfying an LDP on $C([0,T];H_{n})$, we first note that the skeleton equation corresponding to \eqref{eq:SPDEsemi} is given by
\begin{equation}\label{eq:SPDEsemiskeletoneq}
      \frac{\dif{z_{0,y}^{n,\psi}(t)}}{\dif{t}}
      =
      A_{n}z_{0,y}^{n,\psi}(t) + F_{n}(z_{0,y}^{n,\psi}(t)) + Q_{n}^{\frac12}\psi(t),
      \quad t \in (0,T],
      \quad z_{0,y}^{n,\psi}(0) = y
\end{equation}
with $\psi \in L^{2}(0,T;H_{n})$. According to
\cite[Theorem 1.1]{peszat1994large}, we know that \eqref{eq:SPDEsemiskeletoneq} admits a unique mild solution $z_{0,y}^{n,\psi} \in C([0,T];H_{n})$, given by
\begin{equation}\label{eq:skeletoneqspatial}
      z_{0,y}^{n,\psi}(t)
      =
      E_{n}(t)y
      +
      \int_{0}^{t} E_{n}(t-s)F_{n}(z_{0,y}^{n,\psi}(s)) \diff{s}
      +
      \int_{0}^{t} E_{n}(t-s)Q_{n}^{\frac12}\psi(s) \diff{s},
      \quad t \in [0,T].
\end{equation}
With this, we define $I_{0,T}^{n,y} \colon C([0,T];H_{n}) \to [0,+\infty]$ by
\begin{equation}\label{eq:SPDEsemiratefun}
      I_{0,T}^{n,y}(z)
      =
      \frac{1}{2}\inf\{|\psi|_{L^2(0,T;H)}^{2}:
      \psi \in L^2(0,T;H_{n}), z_{0,y}^{n,\psi} = z\},
      \quad z \in C([0,T];H_{n}).
\end{equation}

Following the approach in \cite{peszat1994large}, one can prove the uniform LDP of sample paths of $\{X_{y}^{\varepsilon,n}\}_{\varepsilon > 0}$.

\begin{Theorem}
\label{thm:semisolutionLDP}
      Suppose that Assumptions \ref{ass:AQ} and \ref{ass:F} hold. Then $\{X_{y}^{\varepsilon,n}\}_{\varepsilon > 0}$ satisfies a uniform LDP on $C([0,T];H_{n})$ with the rate $\frac{1}{\varepsilon^{2}}$ and the good rate function $I_{0,T}^{n,y}$ given by \eqref{eq:SPDEsemiratefun}, i.e.,

      \begin{enumerate}
        \item [(i)] compact level set: for any $T > 0$, $n \in \N^{+}$ and $y \in H_{n}$, the level set $K_{0,T}^{n,y}(\alpha) := \{z \in C([0,T];H_{n}): I_{0,T}^{n,y}(z) \leq \alpha\}$ is compact for every $\alpha \geq 0$;
        \item [(ii)] uniform lower bound: for any $T > 0$, $n \in \N^{+}$, $\alpha \geq 0$, $\delta > 0$, $\gamma > 0$ and $K > 0$, there exists $\varepsilon_0 > 0$ such that for any $y \in H_{n}$ with $|y| \leq K$ and $z \in K_{0,T}^{n,y}(\alpha)$,
      \begin{equation}\label{eq:semisoluldplower}
            \P(|X_{y}^{\varepsilon,n}-z|_{C([0,T];H)} < \delta)
            \geq
            \exp\Big(-\frac{I_{0,T}^{n,y}(z)+\gamma}{\varepsilon^2}\Big),
            \quad \varepsilon \leq \varepsilon_{0};
      \end{equation}
        \item [(iii)] uniform upper bound: for any $T > 0$, $n \in \N^{+}$, $\alpha \geq 0$, $\delta > 0$, $\gamma > 0$ and $K > 0$, there exists $\varepsilon_0 > 0$ such that for any $y \in H_{n}$ with $|y| \leq K$,
      \begin{equation}\label{eq:semisoluldpupper}
	        \P(\rho^{C([0,T];H)}
            (X_{y}^{\varepsilon,n},K_{0,T}^{n,y}(\alpha))
            \geq \delta)
	        \leq
	        \exp\Big(-\frac{\alpha-\gamma}{\varepsilon^2}\Big),
	        \quad \varepsilon \leq \varepsilon_{0},
      \end{equation}
      where $\rho^{C([0,T];H)}(z,U) := \inf\limits_{z' \in U}|z-z'|_{C([0,T];H)}, z \in C([0,T];H), U \subset C([0,T];H)$.
      \end{enumerate}
\end{Theorem}

\subsection{Weakly asymptotical preservation for LDP of $\{X_{x}^{\varepsilon}\}_{\varepsilon > 0}$}
\label{subsec:semiweakasympre}

This part aims to estimate the error between the rate functions $I_{0,T}^{n,y}$ and $I_{0,T}^{x}$. To this end, we first simplify the expression of $I_{0,T}^{x}$, which relies on the spatial regularity of the solution of the skeleton equation \eqref{eq:skeletoneqofSPDE}.

\begin{Lemma}
\label{lem:spatialregularity}
      Suppose that Assumptions \ref{ass:AQ} and \ref{ass:F} hold and let $\{z_{0,x}^{\varphi}(t)\}_{t \in [0,T]}$ with $\varphi \in L^{2}(0,T;H)$ be given by \eqref{eq:skeletoneqofSPDE}. If
      $x \in \dot{H}^{2}$, then $z_{0,x}^{\varphi}(t) \in \dot{H}^{2},t \in [0,T]$ and there exists $C > 0$ such that
      \begin{equation}\label{eq:38}
            |z_{0,x}^{\varphi}(t)|_{\dot{H}^{2}}
            \leq
            C(1+|x|_{\dot{H}^{2}}),
            \quad t \in [0,T].
      \end{equation}
\end{Lemma}

\begin{proof}
      Using $|E(t)|_{\mathcal{L}(H)} \leq e^{-\lambda_{1}t} \leq 1,t \geq 0$, \eqref{eq:globalLF} and the H\"{o}lder inequality
      leads to
      \begin{equation*}
      \begin{split}
            &~|z_{0,x}^{\varphi}(t)|
            \leq
            |E(t)x|
            +
            \int_{0}^{t} |E(t-s)F(z_{0,x}^{\varphi}(s))| \diff{s}
            +
            \int_{0}^{t} |E(t-s)Q^{\frac12}\varphi(s)| \diff{s}
            \\\leq&~
            |x|
            +
            C\int_{0}^{t} (1+|z_{0,x}^{\varphi}(s)|) \diff{s}
            +
            |Q^{\frac12}|_{\mathcal{L}(H)} \Big(\int_{0}^{t} 1^{2} \diff{s}\Big)^{\frac{1}{2}}
            \Big(\int_{0}^{t} |\varphi(s)|^{2} \diff{s}\Big)^{\frac{1}{2}}
            \\\leq&~
            |x| + CT
            +
            \sqrt{T}|Q^{\frac12}|_{\mathcal{L}(H)}
            |\varphi|_{L^{2}(0,T;H)}
            +
            C\int_{0}^{t} |z_{0,x}^{\varphi}(s)| \diff{s}.
      \end{split}
      \end{equation*}
      The Gronwall inequality implies
      \begin{equation}\label{eq:39393939}
           |z_{0,x}^{\varphi}(t)|
            \leq
            C(1+|x|), \quad t \in [0,T].
      \end{equation}
      Similarly, we use $|(-A)^{\gamma}E(t)|_{\mathcal{L}(H)}
	  \leq Ct^{-\gamma}, t > 0,\gamma \geq 0$  to get
      \begin{align*}
            &|z_{0,x}^{\varphi}(t)|_{\dot{H}^{1}}
            \leq
            |E(t)x|_{\dot{H}^{1}}
            +
            \int_{0}^{t} |E(t-s)F(z_{0,x}^{\varphi}(s))|_{\dot{H}^{1}} \diff{s}
            +
            \int_{0}^{t} |E(t-s)Q^{\frac12}\varphi(s)|_{\dot{H}^{1}} \diff{s}
            \\\leq&
            |x|_{\dot{H}^{1}}
            +
            C\int_{0}^{t} |(-A)^{\frac{1}{2}}E(t-s)|_{\mathcal{L}(H)}
            |F(z_{0,x}^{\varphi}(s))| \diff{s}
            +
            |(-A)^{\frac{1}{2}}Q^{\frac12}|_{\mathcal{L}(H)}
            \int_{0}^{t} |\varphi(s)| \diff{s}
            \\\leq&
            |x|_{\dot{H}^{1}} + C(1+|x|)
            +
            \sqrt{T}
            |(-A)^{\frac{1}{2}}Q^{\frac12}|_{\mathcal{L}_{2}(H)}
            |\varphi|_{L^{2}(0,T;H)}
            \\\leq&
            C(1+|x|_{\dot{H}^{1}}).
      \end{align*}
      In the same way,
      \begin{align*}
            |z_{0,x}^{\varphi}(t)|_{\dot{H}^{2}}
            \leq&~
            |E(t)x|_{\dot{H}^{2}}
            +
            \int_{0}^{t} |E(t-s)F(z_{0,x}^{\varphi}(s))|_{\dot{H}^{2}} \diff{s}
            +
            \int_{0}^{t} |E(t-s)Q^{\frac12}\varphi(s)|_{\dot{H}^{2}} \diff{s}
            \\\leq&~
            |x|_{\dot{H}^{2}}
            +
            \int_{0}^{t} |(-A)^{\frac{1}{2}}E(t-s)|_{\mathcal{L}(H)}
            |F(z_{0,x}^{\varphi}(s))|_{\dot{H}^{1}} \diff{s}
            \\&~+
            \int_{0}^{t} |(-A)E(t-s)Q^{\frac12}|_{\mathcal{L}(H)}|\varphi(s)| \diff{s}
            \\\leq&~
            |x|_{\dot{H}^{2}}
            +
            C(1+|x|_{\dot{H}^{1}})
            +
            |\varphi|_{L^{2}(0,T;H)}
            \Big(\int_{0}^{t} |(-A)E(t-s)Q^{\frac12}|_{\mathcal{L}_{2}(H)}^{2} \diff{s}\Big)^{\frac{1}{2}}
            \\\leq&~
            |x|_{\dot{H}^{2}} + C(1+|x|_{\dot{H}^{2}})
            +
            C|\varphi|_{L^{2}(0,T;H)} |(-A)^{\frac{1}{2}}Q^{\frac12}|_{\mathcal{L}_{2}(H)}
            \\\leq&~
            C(1+|x|_{\dot{H}^{2}}),
      \end{align*}
      where we have used \eqref{eq:auxiliaryassumptionFuH1} and \cite[Lemma 2.3]{wang2015note}.
      Thus we complete the proof.
\end{proof}

Lemma \ref{lem:spatialregularity} means that
the mild solution $\{z_{0,x}^{\varphi}(t)\}_{t \in [0,T]}$ is also a strong solution. To proceed, we denote by $W_{2}^{1,2}(T)$ the closure of $C^{\infty}([0,T] \times [0,1];\R)$ in the norm
\begin{equation*}
\begin{split}
      |z|_{W_{2}^{1,2}(T)}
      :=&
      \Big(
      \int_{0}^{T}\int_{0}^{1}
          |z(t,\xi)|^{2}
          +
          \Big|\frac{\partial z(t,\xi)}{\partial t}\Big|^{2}
          +
          \Big|\frac{\partial z(t,\xi)}{\partial\xi}\Big|^{2}
          +
          \Big|\frac{\partial^{2} z(t,\xi)}{\partial\xi^{2}}\Big|^{2}
        \diff{\xi}\diff{t}
      \Big)^{\frac{1}{2}}.
\end{split}
\end{equation*}
%
As $Q$ is one-to-one, we have
\begin{equation}\label{eq:I0Txz}
      I_{0,T}^{x}(z)
      =
      \begin{cases}
        \frac{1}{2}
        \int_{0}^{T}
            \Big|Q^{-\frac{1}{2}}
            \Big(\frac{\diff{z(t)}}{\diff{t}}
                 -
                 Az(t) - F(z(t))\Big)
            \Big|^{2}
        \diff{t}
        , & \mbox{if } z \in W_{2,Q}^{1,2}(T) \text{~with~} z(0) = x, \\
        +\infty, & \text{otherwise},
      \end{cases}
\end{equation}
where
\begin{align*}
      W_{2,Q}^{1,2}(T)
      :=
      \Big\{&z \in W_{2}^{1,2}(T):
      \frac{\diff{z(t)}}{\diff{t}} - Az(t) - F(z(t)) \in Q^{\frac{1}{2}}(H) \text{~for almost all~} t \in [0,T],
      \\&\int_{0}^{T}
      \Big|Q^{-\frac{1}{2}}\Big(\frac{\diff{z(t)}}{\diff{t}} - Az(t) - F(z(t))\Big)\Big|^{2}\diff{t} < +\infty\Big\}.
\end{align*}

For the rate function $I_{0,T}^{n,y}$, its effective domain is given by
$$
      \mathcal{H}_{1}^{n,y}(T)
      :=
      \Big\{z \in C([0,T];H_{n}):
      z(\cdot) = y + \int_{0}^{\cdot} \upsilon(s) \diff{s},
      \upsilon \in L^{2}(0,T;H_{n})\Big\}.
$$
It follows that
\begin{equation}\label{eq:I0Tnyz}
      I_{0,T}^{n,y}(z)
      =
      \begin{cases}
        \frac{1}{2}
        \int_{0}^{T}
            \Big|Q_{n}^{-\frac12}\Big(\frac{\diff{z(t)}}{\diff{t}}
                 -
                 A_{n}z(t) - F_{n}(z(t))\Big)
            \Big|^{2}
        \diff{t}
        , & \mbox{if~} z \in \mathcal{H}_{1}^{n,y}(T), \\
        +\infty, & \mbox{otherwise},
      \end{cases}
\end{equation}
where $Q_{n}^{-\frac{1}{2}} \colon H_{n} \to H_{n}$ is the inverse operator of the bijective operator $Q_{n}^{\frac{1}{2}} \colon H_{n} \to H_{n}$.

Similarly to \cite[Definition 4.2]{chen2020large}, we introduce the definition of weakly asymptotical preservation for the LDP of sample paths by a numerical method.
\begin{Definition}
\label{def:semiweakasympres}
      We say that 
      the semi-discrete numerical method \eqref{eq:SPDEsemi}
      weakly asymptotically preserves the LDP of $\{X_{x}^{\varepsilon}\}_{\varepsilon > 0}$ if for any $\kappa > 0$ and $z \in W_{2,Q}^{1,2}(T)$ with $z(0) = x$, there exist $n \in \N^{+}$ and $z_{n} \in \mathcal{H}_{1}^{n,y}(T)$ such that
      $$
            |z-z_{n}|_{C([0,T];H)} < \kappa,
            \quad\quad
            |I_{0,T}^{x}(z)-I_{0,T}^{n,y}(z_{n})| < \kappa.
      $$
\end{Definition}

After these preparations, we have the following result.

\begin{Theorem}
\label{thm:spatialsolutionwap}
      Suppose that Assumptions \ref{ass:AQ} and \ref{ass:F} hold. If $x \in \dot{H}^{2}$, then
      the semi-discrete numerical method \eqref{eq:SPDEsemi}
      weakly asymptotically preserves the LDP of $\{X_{x}^{\varepsilon}\}_{\varepsilon > 0}$.
\end{Theorem}

\begin{proof}
      For any $z \in W_{2,Q}^{1,2}(T)$ with $z(0) = x$,
      we define
      $$\varphi(t)
      :=
      Q^{-\frac{1}{2}}
            \Big(\frac{\diff{z(t)}}{\diff{t}}
                 -
                 Az(t) - F(z(t))\Big)$$
      for almost all $t \in [0,T]$. Then $\varphi \in L^2(0,T;H)$,  $z_{0,x}^{\varphi} = z$ and
      \begin{equation}\label{eq:IoTxz311311}
            \frac{1}{2}|\varphi|_{L^2(0,T;H)}^{2} = I_{0,T}^{x}(z) < +\infty,
      \end{equation}
      which together with \eqref{eq:39393939} in Lemma \ref{lem:spatialregularity} yields
      \begin{equation}\label{eq:ztdotH2z0x313}
            |z(t)|
            =
            |z_{0,x}^{\varphi}(t)|
            \leq
            C(1+|x|),\quad t \in [0,T].
      \end{equation}
      Now for any $n \in \N^{+}$, we define $\{z_{n}(t)\}_{t \in [0,T]}$ by
      \begin{equation}\label{eq:dzntdt312}
            \frac{\diff{z_{n}(t)}}{\diff{t}}
            =
            A_{n}z_{n}(t) + F_{n}(z_{n}(t))
            + Q_{n}^{\frac{1}{2}}P_{n}\varphi(t),
            \quad t \in (0,T],
            \quad z_{n}(0) = P_{n}x.
      \end{equation}
      It follows from $\langle Au,u \rangle \leq -\lambda_{1}|u|^{2}, u \in H$ and \eqref{eq:globalLF} that
      \begin{align*}
            &\frac{1}{2}\frac{\diff{|z(t)-z_{n}(t)|^{2}}}{\diff{t}}
            =
            \Big\langle \frac{\diff{(z(t)-z_{n}(t))}}{\diff{t}},
            z(t)-z_{n}(t) \Big\rangle
            \\=&
            \langle A(z(t)-z_{n}(t)),z(t)-z_{n}(t) \rangle
            +
            \langle (I-P_{n})F(z(t)),z(t)-z_{n}(t) \rangle
            \\&+
            \langle F_{n}(z(t))-F_{n}(z_{n}(t)),z(t)-z_{n}(t) \rangle
            +
            \langle (I-P_{n})Q^{\frac{1}{2}}\varphi(t),
            z(t)-z_{n}(t) \rangle
            \\\leq&
            -\lambda_{1}|z(t)-z_{n}(t)|^{2}
            +
            |(I-P_{n})F(z(t))||z(t)-z_{n}(t)|
            \\&+
            L_{F}|z(t)-z_{n}(t)|^{2}
            +
            |(I-P_{n})Q^{\frac{1}{2}}\varphi(t)|
            |z(t)-z_{n}(t)|
            \\\leq&
            L_{F}|z(t)-z_{n}(t)|^{2}
            +
            \frac{1}{2\lambda_{1}}|(I-P_{n})F(z(t))|^{2}
            +
            \frac{1}{2\lambda_{1}}
            |(I-P_{n})Q^{\frac{1}{2}}\varphi(t)|^{2},
      \end{align*}
      where the weighted Young inequality is used in the last step. Integrating yields
      \begin{align*}
            |z(t)&-z_{n}(t)|^{2}
            \leq
            |x-P_{n}x|^{2}
            +
            2L_{F}\int_{0}^{t}|z(s)-z_{n}(s)|^{2}\diff{s}
            \\&+
            \frac{1}{\lambda_{1}}\int_{0}^{t}
            |(I-P_{n})F(z(s))|^{2}\diff{s}
            +
            \frac{1}{\lambda_{1}}\int_{0}^{t}
            |(I-P_{n})Q^{\frac{1}{2}}\varphi(s)|^{2}\diff{s}
      \end{align*}
      and consequently
      \begin{align*}
            |z&-z_{n}|_{C([0,t];H)}^{2}
            \leq
            \Pi_{n}
            +
            2L_{F}\int_{0}^{t}|z-z_{n}|_{C([0,s];H)}^{2}\diff{s}
      \end{align*}
      with
      \begin{equation}
            \Pi_{n}
            :=
            |(I-P_{n})x|^{2}
            +
            \frac{1}{\lambda_{1}}\int_{0}^{T}
            |(I-P_{n})F(z(s))|^{2}\diff{s}
            +
            \frac{1}{\lambda_{1}}\int_{0}^{T}
            |(I-P_{n})Q^{\frac{1}{2}}\varphi(s)|^{2}\diff{s}.
      \end{equation}
      The Gronwall inequality shows
      \begin{equation}\label{eq:zzn314}
            |z-z_{n}|_{C([0,t];H)}^{2} \leq \Pi_{n}e^{2L_{F}t} \leq \Pi_{n}e^{2L_{F}T}, \quad t \in [0,T].
      \end{equation}
      Applying \eqref{eq:globalLF} and \eqref{eq:ztdotH2z0x313} yields $$|(I-P_{n})F(z(t))|^{2} \leq |F(z(t))|^{2} = |F(z_{0,x}^{\varphi}(t))|^{2} \leq C(1+|x|^{2}),\quad t \in [0,T].$$
      We use $\lim\limits_{n \to \infty}|(I-P_{n})F(z(t))|^{2} = 0, t \in [0,T]$ and the bounded convergence theorem to get
      \begin{equation*}
            \lim_{n \to \infty}
            \frac{1}{\lambda_{1}}\int_{0}^{T}
            |(I-P_{n})F(z(s))|^{2}\diff{s}
            =
            0.
      \end{equation*}
      In the same way, one can validate $\lim\limits_{n \to \infty}\Pi_{n} = 0$ and thus $\lim\limits_{n \to \infty}|z-z_{n}|_{C([0,T];H)} = 0$ by \eqref{eq:zzn314}. It follows that for any $\kappa > 0$, there exists $n_{1} \in \N^{+}$ such that
      \begin{equation}\label{eq:z-znC0TH}
            |z-z_{n}|_{C([0,T];H)} < \kappa,\quad n > n_{1}.
      \end{equation}
      By \eqref{eq:dzntdt312} and \eqref{eq:I0Tnyz}, we have $z_{n} \in \mathcal{H}_{1}^{n,y}(T)$ and hence
      \begin{equation*}
            I_{0,T}^{n,y}(z_{n})
            =
            \frac{1}{2}
            \int_{0}^{T}
            \Big|Q_{n}^{-\frac12}
            \Big(\frac{\diff{z_{n}(t)}}{\diff{t}}
                 -
                 A_{n}z_{n}(t) - F_{n}(z_{n}(t))\Big)
            \Big|^{2}
            \diff{t}.
      \end{equation*}
      It immediately follows from \eqref{eq:dzntdt312} that
      \begin{equation}\label{eq:I0Tny317317}
            I_{0,T}^{n,y}(z_{n})
            =
            \frac{1}{2}|P_{n}\varphi|_{L^2(0,T;H)}^{2}.
      \end{equation}
      Applying \eqref{eq:IoTxz311311}, \eqref{eq:I0Tny317317} and the H\"{o}lder inequality leads to
      \begin{align*}
             |I_{0,T}^{x}(z)-I_{0,T}^{n,y}(z_{n})|
            =&
            \frac{1}{2}\Big|\int_{0}^{T}
            \big\langle (I+P_{n})\varphi(t),
            (I-P_{n})\varphi(t)\big\rangle \diff{t}\Big|
            \\\leq&
            |\varphi|_{L^2(0,T;H)}
            \Big(\int_{0}^{T}
            |(I-P_{n})\varphi(t)|^{2} \diff{t}\Big)^{\frac{1}{2}}.
      \end{align*}
      By the Lebesgue dominated convergence theorem, we conclude
      \begin{equation}\label{eq:313313313}
            \lim\limits_{n \to \infty}
            |I_{0,T}^{x}(z)-I_{0,T}^{n,y}(z_{n})| = 0,
      \end{equation}
      which yields that there exists $n_{2} \in \N^{+}$ such that
      \begin{equation}\label{eq:I0Txz-I0Tnyzn}
            |I_{0,T}^{x}(z)-I_{0,T}^{n,y}(z_{n})|
            <
            \kappa, \quad n > n_{2}.
      \end{equation}
      By choosing $n > \max\{n_{1},n_{2}\}$, we complete the proof.
\end{proof}

\subsection{LDP for invariant measures of spatial discretization}
\label{subsec:semiinvameaLDPs}
As is well known, the uniform LDP of sample paths not only is of importance in itself but also can be used to derive the LDP of invariant measures; see, e.g., \cite{sowers1992largeinvariant,cerrai2005largeinvariant}. With this purpose, one must first guarantee the existence of invariant measures of spatial discretization \eqref{eq:SPDEsemi}.
According to \cite[Theorem 3.1]{CGW2020ErgodicityANM} and \cite[Remark 7.2 and Proposition 7.10]{da2006introduction},
the family of probability measures $\{\mu_{t}^{\varepsilon,n}\}_{t > 0}$, defined by
      $$
            \mu_{t}^{\varepsilon,n}(B)
            :=
            \frac{1}{t} \int_{0}^{t}
            \P\big(X_{0}^{\varepsilon,n}(s) \in B\big)
            \diff{s},
            \quad B \in \mathcal{B}(H_{n}), t > 0
      $$
is tight. It follows from the Krylov--Bogoliubov theorem (\cite[Theorem 7.1]{da2006introduction}) that there exists $\{t_{i}\}_{i \in \N^{+}} \uparrow +\infty$ (possibly depending on $\varepsilon$) such that the sequence $\{\mu_{t_{i}}^{\varepsilon,n}\}_{i \in \N^{+}}$ converges weakly to some probability measure $\mu^{\varepsilon,n}$ on $(H_{n},\mathcal{B}(H_{n}))$, which is invariant for \eqref{eq:SPDEsemi}. 
To study the LDP of invariant measures $\{\mu^{\varepsilon,n}\}_{\varepsilon > 0}$, we define $V^{n} \colon H_{n} \to [0,+\infty]$ by
\begin{equation}\label{eq:Vny}
      V^{n}(v)
      =
      \inf\{I_{0,T}^{n,0}(z) : T > 0, z \in C([0,T];H_{n}), z(0) = 0, z(T) = v \},
\end{equation}
and its level set
$K^{n}(\alpha) = \{v \in H_{n} : V^{n}(v) \leq \alpha\}$ for each $\alpha \geq 0$.

\subsubsection{$V^{n}$ being a good rate function}

In this part, we will use the boundedness of $K^{n}$ to show that $V^{n}$ is a good rate function. To this end, we need the following lemma.

\begin{Lemma}\label{lem:4.9}
      Suppose that Assumptions \ref{ass:AQ}, \ref{ass:F} and \ref{ass:LFleqlambda1} hold and let $\{z(t)\}_{t \in [0,T]}$ be the mild solution of the following equation
      \begin{equation}\label{eq:diffztdifft}
            \frac{\diff{z(t)}}{\diff{t}}
            =
            A_{n}z(t) + F_{n}(z(t)) + Q_{n}^{\frac12}\psi(t),
            \quad t \in (0,T],
            \quad z(0) = y \in H_{n}
      \end{equation}
      with $\psi \in L^{2}(0,T;H_{n})$. Then we have
      \begin{equation}\label{eq:ztHn2}
            |z(t)|^{2}
            \leq
            C\big(1+|y|^{2}
            +
            |Q^{\frac12}|_{\mathcal{L}(H)}^{2}
            |\psi|_{L^{2}(0,T;H)}^{2}\big),
            \quad t \in [0,T].
      \end{equation}
\end{Lemma}

\begin{proof}
      Setting $O(t) :=  \int_{0}^{t} E_{n}(t-s) Q_{n}^{\frac12}\psi(s) \diff{s}, t \in [0,T]$, we use $|E_{n}(t)|_{\mathcal{L}(H_{n})} \leq e^{-\lambda_{1}t},t \in [0,T]$ and the H\"{o}lder inequality to get
      \begin{equation}\label{eq:OtHn}
      \begin{split}
            |O(t)|
            \leq&
            |Q_{n}^{\frac12}|_{\mathcal{L}(H)}
            \int_{0}^{t} e^{-\lambda_{1}(t-s)}|\psi(s)| \diff{s}
            \leq
            |Q^{\frac12}|_{\mathcal{L}(H)}
            |\psi|_{L^{2}(0,T;H)},
            \quad t \in [0,T].
      \end{split}
      \end{equation}
      Let $\bar{z}(t) :=  z(t)-O(t), t \in [0,T]$, then $\{\bar{z}(t)\}_{t \in [0,T]}$ satisfies
      $$
            \frac{\dif{\bar{z}(t)}}{\dif{t}}
            =
            A_{n}\bar{z}(t) + F_{n}(\bar{z}(t)+O(t)),
            \quad t \in (0,T],
            \quad \bar{z}(0) = y.
      $$
      Using \eqref{eq:dissipativecond} 
      and \eqref{eq:globalLF}
      yields
      \begin{align*}
            &~\frac{\diff{e^{ct}|\bar{z}(t)|^{2}}}
                 {\diff{t}}
            =
            ce^{ct}|\bar{z}(t)|^{2}
            +
            2e^{ct}\big\langle
                 A_{n}\bar{z}(t) + F_{n}(\bar{z}(t)+O(t)),\bar{z}(t)
            \big\rangle
            \\=&~
            ce^{ct}|\bar{z}(t)|^{2}
            +
            2e^{ct}\big\langle
                 A_{n}\bar{z}(t) + F_{n}(\bar{z}(t)),\bar{z}(t)
            \big\rangle
            +
            2e^{ct}\big\langle
                 F_{n}(\bar{z}(t)+O(t))-F_{n}(\bar{z}(t)),\bar{z}(t)
            \big\rangle
            \\\leq&~
            -
            ce^{ct}|\bar{z}(t)|^{2}
            +
            2L_{F}e^{ct}|O(t)||\bar{z}(t)|
            \leq
            \frac{L_{F}^{2}}{c}e^{ct}|O(t)|^{2}
      \end{align*}
      due to the weighted Young inequality $ab \leq \kappa a^{2} +\frac{b^{2}}{4\kappa}$ for all $a,b \in \R$ with $\kappa = c > 0$ in the last step. Applying \eqref{eq:OtHn} leads to
      \begin{equation*}
            |\bar{z}(t)|^{2}
            \leq
            |y|^{2}
            +
            \frac{L_{F}^{2}}{c^{2}}
            |Q^{\frac12}|_{\mathcal{L}(H)}^{2}
            |\psi|_{L^{2}(0,T;H)}^{2},
            \quad t \in [0,T].
      \end{equation*}
      Together with $|z(t)|^{2} \leq 2|\bar{z}(t)|^{2} + 2|O(t)|^{2}$ and \eqref{eq:OtHn}, we complete the proof.
\end{proof}

\begin{Theorem}
\label{th:semiinvariantgood}
      Suppose that Assumptions \ref{ass:AQ}, \ref{ass:F} and \ref{ass:LFleqlambda1} hold. Then the mapping $V^{n}$ given by \eqref{eq:Vny} is a good rate function.
\end{Theorem}

\begin{proof}
      Notice that $V^{n}$ is lower semi-continuous and $H_{n}$ is a finite dimensional space. In order to verify $V^{n}$ being a good rate function, it suffices to show that $K^{n}(\alpha)$ is bounded. For any $y \in K^{n}(\alpha)$, the definition of $V^{n}(y)$ implies that for any $\kappa > 0$ there exists $T_{\kappa} > 0, z_{\kappa} \in C([0,T_{\kappa}];H_{n}), z_{\kappa}(0) = 0, z_{\kappa}(T_{\kappa}) = y$ such that
      \begin{equation*}
            I_{0,T_{\kappa}}^{n,0}(z_{\kappa})
            \leq
            V^{n}(y) + \kappa
            \leq
            \alpha + \kappa.
      \end{equation*}
      The definition of $I_{0,T_{\kappa}}^{n,0}(z_{\kappa})$ means that there exists $\psi_{\kappa} \in L^{2}(0,T_{\kappa};H_{n})$ with $z_{0,0}^{n,\psi_{\kappa}} = z_{\kappa}$ such that
      \begin{equation*}
            \frac{1}{2}|\psi_{\kappa}|_{L^{2}(0,T_{\kappa};H)}^{2}
            \leq
            I_{0,T_{\kappa}}^{n,0}(z_{\kappa}) + \kappa
            \leq
            \alpha + 2\kappa.
      \end{equation*}
      Because of $z_{0,0}^{n,\psi_{\kappa}} = z_{\kappa}$ satisfying \eqref{eq:diffztdifft}, we use Lemma \ref{lem:4.9} to get
      \begin{equation*}
            |y|^{2}
            =
            |z_{\kappa}(T_{\kappa})|^{2}
            \leq
            C\big(1
            +
            |Q^{\frac12}|_{\mathcal{L}(H)}^{2}
            |\psi_{\kappa}|_{L^{2}(0,T_{\kappa};H)}^{2}\big)
            \leq
            C\big(1
            +
            2(\alpha+2\kappa)
            |Q^{\frac12}|_{\mathcal{L}(H)}^{2}\big).
      \end{equation*}
      Thus we complete the proof.
\end{proof}

\subsubsection{Lower bound estimate for the LDP of $\{\mu^{\varepsilon,n}\}_{\varepsilon > 0}$}

\begin{Theorem}
\label{th:semiinvariantmeasurelowerbounded}
      Suppose that Assumptions \ref{ass:AQ}, \ref{ass:F} and \ref{ass:LFleqlambda1} hold. Then for any $n \in \N^{+}$, $\bar{y} \in H_{n}$, $\delta > 0$ and $\gamma > 0$, there exists $\varepsilon_{0} > 0$ such that
      \begin{equation}\label{eq:muvnyHndelta}
            \mu^{\varepsilon,n}(\{y \in H_{n} : |y-\bar{y}| < \delta\})
            \geq
            \exp\Big(-\frac{V^{n}(\bar{y})+\gamma}{\varepsilon^{2}}\Big),
            \quad \varepsilon \leq \varepsilon_{0}.
      \end{equation}
\end{Theorem}

\begin{proof}
      Without loss of generality, we assume  $V^{n}(\bar{y}) < +\infty$, otherwise \eqref{eq:muvnyHndelta} obviously holds. For any $\bar{y} \in H_{n}$ and $\gamma > 0$, the definition of $V^{n}(\bar{y})$ implies that there exists $\bar{T} > 0$, $\bar{z} \in C([0,\bar{T}];H_{n})$ with $\bar{z}(0) = 0$ and $\bar{z}(\bar{T}) = \bar{y}$ such that
      $
            I_{0,\bar{T}}^{n,0}(\bar{z}) \leq V^{n}(\bar{y}) + \frac{\gamma}{4}.
      $
      The definition of $I_{0,\bar{T}}^{n,0}(\bar{z})$
      further implies that there exists $\bar{\varphi} \in L^{2}(0,\bar{T};H_{n})$ with $z_{0,0}^{n,\bar{\varphi}} = \bar{z}$ such that
      $$
            \frac{1}{2}|\bar{\varphi}|_{L^{2}(0,\bar{T};H)}^{2}
            \leq
            I_{0,\bar{T}}^{n,0}(\bar{z}) + \frac{\gamma}{4}
            \leq
            V^{n}(\bar{y}) + \frac{\gamma}{2}.
      $$
      For any $T > \bar{T}$, $K > 0$ and $y \in H_{n}$ with $|y| \leq K$, we define
      $$
            \tilde{\varphi}(t)
            :=
            \left\{
            \begin{array}{ll}
              0, & t \in [0,T-\bar{T}], \\
              \bar{\varphi}(t-(T-\bar{T})), & t \in [T-\bar{T},T]
            \end{array}
            \right.
      $$
      and
      $$
            \tilde{z}(t)
            :=
            z_{0,y}^{n,\tilde{\varphi}}(t)
            =
            \left\{
            \begin{array}{ll}
              z_{0,y}^{n,0}(t), &  t \in [0,T-\bar{T}], \\
              z_{T-\bar{T},z_{0,y}^{n,0}(T-\bar{T})}^{n,\bar{\varphi}(\cdot-(T-\bar{T}))}(t),
              & t \in [T-\bar{T},T].
            \end{array}
            \right.
      $$
      Then 
      we have
      $|\tilde{\varphi}|_{L^{2}(0,T;H)}^{2}
       =
       |\bar{\varphi}|_{L^{2}(0,\bar{T};H)}^{2}$
      ,
      $\tilde{z} \in C([0,T];H_{n})$ and
      \begin{equation}\label{eq:I0Tnytildez}
      \begin{split}
            I_{0,T}^{n,y}(\tilde{z})
            =&
            I_{0,T}^{n,y}(z_{0,y}^{n,\tilde{\varphi}})
            =
            \frac{1}{2}\inf\{|\varphi|_{L^{2}(0,T;H)}^{2}:
            \varphi \in L^{2}(0,T;H_{n}), z_{0,y}^{n,\varphi} = \tilde{z}\}
            \\\leq&
            \frac{1}{2}|\tilde{\varphi}|_{L^{2}(0,T;H)}^{2}
            =
            \frac{1}{2}|\bar{\varphi}|_{L^{2}(0,\bar{T};H)}^{2}
            \leq
            V^{n}(\bar{y}) + \frac{\gamma}{2}.
      \end{split}
      \end{equation}
      In view of
      \begin{equation*}
      \begin{split}
            \tilde{z}(t)
            =&
            E_{n}(t-(T-\bar{T}))z_{0,y}^{n,0}(T-\bar{T})
            +
            \int_{0}^{t-(T-\bar{T})} E_{n}(t-(T-\bar{T})-s)
            F_{n}(\tilde{z}(s+T-\bar{T})) \diff{s}
            \\&+
            \int_{0}^{t-(T-\bar{T})} E_{n}(t-(T-\bar{T})-s)
            Q_{n}^{\frac12} \bar{\varphi}(s) \diff{s},
            \quad t \in [T-\bar{T},T],
      \end{split}
      \end{equation*}
      we set $\hat{z}(t) :=  \tilde{z}(t+(T-\bar{T})), t \in [0,\bar{T}]$ to get
      \begin{equation*}
      \begin{split}
            \hat{z}(t)
            =&
            E_{n}(t)z_{0,y}^{n,0}(T-\bar{T})
            +
            \int_{0}^{t} E_{n}(t-s) F_{n}(\hat{z}(s)) \diff{s}
            +
            \int_{0}^{t} E_{n}(t-s)Q_{n}^{\frac12}\bar{\varphi}(s) \diff{s},
            \quad t \in [0,\bar{T}]
      \end{split}
      \end{equation*}
      and consequently
      \begin{equation*}
            \hat{z}(t)-z_{0,0}^{n,\bar{\varphi}}(t)
            =
            E_{n}(t)z_{0,y}^{n,0}(T-\bar{T})
            +
            \int_{0}^{t} E_{n}(t-s)
                \big( F_{n}(\hat{z}(s)) - F_{n}(z_{0,0}^{n,\bar{\varphi}}(s)) \big)
            \diff{s},
            \quad t \in [0,\bar{T}].
      \end{equation*}
      Using $|E_{n}(t)|_{\mathcal{L}(H_{n})} \leq e^{-\lambda_{1}t}$ for all $t \geq 0$, \eqref{eq:globalLF} and the Gronwall inequality yields
      \begin{equation}\label{eq:hatztz00nbar}
      \begin{split}
            |\hat{z}(t)-z_{0,0}^{n,\bar{\varphi}}(t)|
            \leq&
            |z_{0,y}^{n,0}(T-\bar{T})| e^{L\bar{T}},
            \quad t \in [0,\bar{T}].
      \end{split}
      \end{equation}
      As $\{z_{0,y}^{n,0}(t)\}_{t \in [0,T-\bar{T}]}$ is the solution of
      \begin{equation*}
            \frac{\diff{z_{0,y}^{n,0}(t)}}{\diff{t}}
            =
            A_{n}z_{0,y}^{n,0}(t)
            +
            F_{n}(z_{0,y}^{n,0}(t)),
            \quad t \in (0,T-\bar{T}],
            \quad z_{0,y}^{n,0}(0) = y,
      \end{equation*}
      we use \eqref{eq:dissipativecond}
      to get
      \begin{equation*}
      \begin{split}
            \frac{\diff{e^{2ct}|z_{0,y}^{n,0}(t)|^{2}}}{\diff{t}}
            =&
            2ce^{2ct}|z_{0,y}^{n,0}(t)|^{2}
            +
            2e^{2ct}\big\langle
            Az_{0,y}^{n,0}(t) + F(z_{0,y}^{n,0}(t)),
            z_{0,y}^{n,0}(t)\big\rangle
            \leq
            0,
      \end{split}
      \end{equation*}
      which yields $|z_{0,y}^{n,0}(t)| \leq e^{-ct}|y|$ for all $t \in [0,T-\bar{T}]$ and thus
      $\lim\limits_{T \to \infty}
      \sup\limits_{|y|_{H_{n}} \leq K} |z_{0,y}^{n,0}(T-\bar{T})| = 0$.
      Noting \eqref{eq:hatztz00nbar}, $\hat{z}(\bar{T}) = \tilde{z}(T) = z_{0,y}^{n,\tilde{\varphi}}(T)$ and $z_{0,0}^{n,\bar{\varphi}}(\bar{T}) = \bar{z}(\bar{T}) = \bar{y}$, there exists $\tilde{T} > \bar{T}$ such that
      $$
            \sup_{|y| \leq K}
            |z_{0,y}^{n,\tilde{\varphi}}(\tilde{T})-\bar{y}|
            \leq
            \frac{\delta}{2}.
      $$
       This together with the invariance of $\mu^{\varepsilon,n}$ implies
      \begin{equation*}
      \begin{split}
            \mu^{\varepsilon,n}(\{y \in H_{n} : |y-&\bar{y}| < \delta\})
            =
            \int_{H_{n}} \P(|X_{y}^{\varepsilon,n}(\tilde{T})-\bar{y}| < \delta)
            \,\mu^{\varepsilon,n}(\dif{y})
            \\\geq&
            \int_{{|y| \leq K}} \P(|X_{y}^{\varepsilon,n}(\tilde{T})-\bar{y}| < \delta)
            \,\mu^{\varepsilon,n}(\dif{y})
            \\\geq&
            \int_{{|y| \leq K}} \P(|X_{y}^{\varepsilon,n}(\tilde{T})
            - z_{0,y}^{n,\tilde{\varphi}}(\tilde{T})| < \delta/2)
            \,\mu^{\varepsilon,n}(\dif{y})
            \\\geq&
            \int_{{|y| \leq K}}
                \P(|X_{y}^{\varepsilon,n}
                -z_{0,y}^{n,\tilde{\varphi}}|_{C([0,\tilde{T}];H)}
                < \delta/2)
            \,\mu^{\varepsilon,n}(\dif{y}).
      \end{split}
      \end{equation*}
      By \eqref{eq:semisoluldplower} and \eqref{eq:I0Tnytildez}, there exists $\varepsilon_{1} > 0$ such that
      \begin{equation}\label{eq:muvnyHnybary}
      \begin{split}
            \mu^{\varepsilon,n}(\{y \in H_{n} : |y-\bar{y}| < \delta\})
            \geq&~
            \int_{{|y| \leq K}}
                \exp\Big(-\frac{I_{0,\tilde{T}}^{n,y}(z_{0,y}^{n,\tilde{\varphi}})+\gamma/2}
                    {\varepsilon^{2}}\Big)
            \,\mu^{\varepsilon,n}(\dif{y})
            \\\geq&~
            \mu^{\varepsilon,n}({|y| \leq K})
            \exp\Big(-\frac{V^{n}(\bar{y})+\gamma}{\varepsilon^{2}}\Big),
            \quad \varepsilon \leq \varepsilon_{1}.
      \end{split}
      \end{equation}
      It remains to estimate $\mu^{\varepsilon,n}({|y| \leq K})$. Using the definition of $\mu^{\varepsilon,n}$ yields
      \begin{equation}\label{eq:muvarepsilonnyHnK}
            \mu^{\varepsilon,n}(|y| > K)
            =
            \lim_{i \to \infty}
            \frac{1}{t_{i}} \int_{0}^{t_{i}}
            \P(|X_{0}^{\varepsilon,n}(t)| > K) \diff{t}.
      \end{equation}
      Similarly to the proof of Lemma \ref{lem:4.9}, we use \eqref{eq:SPDEsemi} to obtain
      \begin{equation*}
            |X_{0}^{\varepsilon,n}(t)|
            \leq
            \varepsilon\frac{L+c}{c}
            |\Gamma^{n}|_{C([0,t];H)}
      \end{equation*}
      with
      \begin{equation}\label{eq:Gammant}
            \Gamma^{n}(t)
            :=
            \int_{0}^{t}
                E_{n}(t-s)Q_{n}^{\frac12}
            \diff{W_{n}(s)},
            \quad t \geq 0.
      \end{equation}
      This together with the Chebyshev inequality shows that for any $t \geq 0$ and $K > 0$,
      \begin{equation}\label{eq:PX0vntHnK}
      \begin{split}
            &\P(|X_{0}^{\varepsilon,n}(t)| > K)
            \leq
            \P\Big(
            |\Gamma^{n}|_{C([0,t];H)}
            >
            \frac{cK}{\varepsilon(L+c)}\Big)
            \leq
            \varepsilon^{2}\Big(\frac{L+c}{cK}\Big)^{2}C
      \end{split}
      \end{equation}
      with $C$ being independent of $t$, which is due to
      \begin{equation}\label{eq:GammanC0tHn2}
      \begin{split}
            \E\big[|\Gamma^{n}|_{C([0,t];H)}^{2}\big]
            =&~
            \E\Big[\sup_{s \in [0,t]}
            \Big|\int_{0}^{s}
                     E_{n}(s-r)Q_{n}^{\frac12}
                 \diff{W_{n}(r)}
            \Big|^{2}\Big]
            \leq
            C\int_{0}^{t}
                 |E_{n}(t-r)Q_{n}^{\frac12}\big|_{\mathcal{L}_{2}(H_{n})}^{2}
            \diff{r}
            \\\leq&~
            C|(-A_{n})^{-\frac{1}{2}}|_{\mathcal{L}(H_{n})}^{2}
            |(-A_{n})^{\frac{1}{2}}Q_{n}^{\frac12}|_{\mathcal{L}_{2}(H_{n})}^{2}
            \int_{0}^{t}
                |E_{n}(t-r)|_{\mathcal{L}(H_{n})}^{2}
            \diff{r}
            \\\leq&~
            C\int_{0}^{t}
                e^{-2\lambda_{1}(t-r)}
            \diff{r}
            \leq
            C,
            \quad t \geq 0.
      \end{split}
      \end{equation}
      By setting $\bar{K} > 0$, \eqref{eq:muvarepsilonnyHnK} and \eqref{eq:PX0vntHnK} lead to
      \begin{equation*}
            \lim_{\varepsilon \to 0}
            \mu^{\varepsilon,n}(|y| > \bar{K})
            \leq
            \lim_{\varepsilon \to 0}
            \sup_{t \geq 0}\P(|X_{0}^{\varepsilon,n}(t)| > \bar{K})
            \leq
            0
      \end{equation*}
      and consequently
      $
            \lim\limits_{\varepsilon \to 0}
            \mu^{\varepsilon,n}({|y| \leq \bar{K}})
            =
            1
      $.
      This in combination with \eqref{eq:muvnyHnybary} shows that there exists sufficiently small $\varepsilon_{0} \leq \varepsilon_{1}$  such that \eqref{eq:muvnyHndelta} holds. Thus we complete the proof.
\end{proof}

\subsubsection{Upper bound estimate for the LDP of $\{\mu^{\varepsilon,n}\}_{\varepsilon > 0}$}

The following lemma gives the exponential tail estimate for invariant measure $\mu^{\varepsilon,n}$ based on the Fernique theorem.

\begin{Lemma}\label{lem:stochasticevolutiontailestimate}
      Suppose that Assumptions \ref{ass:AQ}, \ref{ass:F} and \ref{ass:LFleqlambda1} hold. Then for any $\alpha \geq 0$, there exists $
      \bar{K} > 0$ such that
      \begin{equation}\label{eq:327327}
            \mu^{\varepsilon,n}(\{u \in H_{n} : |u| > \bar{K}\})
            \leq
            \exp\Big(-\frac{\alpha}{\varepsilon^{2}}\Big),
            \quad \varepsilon \leq 1.
      \end{equation}
\end{Lemma}

\begin{proof}
      For any $K > 0$, it follows from the definition of $\mu^{\varepsilon,n}$ that
      \begin{equation}\label{eq:mvnHnyHnK}
            \mu^{\varepsilon,n}(\{u \in H_{n} : |u| > K\})
            =
            \lim_{i \to \infty}\frac{1}{t_{i}}
                \int_{0}^{t_{i}} \P(|X_{0}^{\varepsilon,n}(t)| > K) \diff{t}.
      \end{equation}
      Similarly to the proof of Lemma \ref{lem:4.9}, one can
      show that there exists $C > 0$ independent of $t$ such that
      $
            |X_{0}^{\varepsilon,n}(t)|
            \leq
            \varepsilon C |\Gamma^{n}(t)|,
      $
      where $\Gamma^{n}(t)$ is defined by \eqref{eq:Gammant}.
      Thus for any $K > 0$ and $\kappa > 0$, we use the Markov inequality to get
      \begin{equation}\label{eq:PBigsup0leqsleqt}
      \begin{split}
            \P(|X_{0}^{\varepsilon,n}(t)| > K)
            \leq&
            \P\Big(|\Gamma^{n}(t)|
            > \frac{K}{\varepsilon C}\Big)
            \\=&
            \P\Big(
            \exp\big(\kappa|\Gamma^{n}(t)|^{2}\big)
            >
            \exp\Big(\kappa\Big(\frac{K}{\varepsilon C}\Big)^{2}\Big)\Big)
            \\\leq&
            \exp\Big(-\kappa\Big(\frac{K}{\varepsilon C}\Big)^{2}\Big)
            \E\big[\exp\big(\kappa|\Gamma^{n}(t)|^{2}\big)\big],
            \quad t \geq 0.
      \end{split}
      \end{equation}
      By \cite[Proposition 4.28]{da2014stochastic},
      \begin{equation*}
            \Gamma^{n}(t)
            \sim
            \mathcal{N}\Big(0,\int_{0}^{t} E_{n}(t-r)Q_{n}E_{n}(t-r)\diff{r}\Big)
      \end{equation*}
      and for any $i = 1,\ldots,n$,
      \begin{equation*}
            \Big(\int_{0}^{t} E_{n}(t-r)Q_{n}E_{n}(t-r)\diff{r}\Big)e_{i}
            =
            \Big(\int_{0}^{t} e^{-2\lambda_{i}(t-r)}q_{i}\diff{r}\Big)e_{i}
            =
            \frac{q_{i}}{2\lambda_{i}}
            \big(1 - e^{-2\lambda_{i}t}\big)e_{i}.
      \end{equation*}
      Applying Fernique's theorem (\cite[Proposition 1.13]{da2006introduction}) yields that for any $\kappa < \min\big\{\frac{\lambda_{i}}{q_{i}}:i =1, \ldots,n\big\}$,
      \begin{equation}\label{eq:EBigexpBigkappa}
      \begin{split}
            &
            \E\Big[
            \exp\big(\kappa|\Gamma^{n}(t)|^{2}\big)
            \Big]
            =
            \Big(\prod_{i=1}^{n}
            \Big(1 - \kappa \frac{q_{i}}{\lambda_{i}}
            \big(1 - e^{-2\lambda_{i}t}\big) \Big)\Big)^{-\frac{1}{2}}
            \leq
            \Big(\prod_{i=1}^{n}
            \Big(1 - \kappa \frac{q_{i}}{\lambda_{i}}\Big)\Big)^{-\frac{1}{2}},
            \quad t \geq 0.
      \end{split}
      \end{equation}
      Combining \eqref{eq:mvnHnyHnK}, \eqref{eq:PBigsup0leqsleqt} and \eqref{eq:EBigexpBigkappa} implies that one can choose
      \begin{equation*}
            \bar{K}
            \geq
            C\sqrt{\frac{1}{\kappa}\Big(\alpha
            + \varepsilon^{2}\ln \Big(\prod_{i=1}^{n}
            \Big(1 - \kappa \frac{q_{i}}{\lambda_{i}}\Big)\Big)^{-\frac{1}{2}}\Big)}
      \end{equation*}
      such that \eqref{eq:327327} holds. Thus we complete the proof.
\end{proof}

For any $T, L > 0$, we define
\begin{equation}\label{eq:K0TnLalpha}
      K_{0,T}^{n,L}(\alpha)
      :=
      \{z \in C([0,T];H_{n}) :
      |z(0)| \leq L,I_{0,T}^{n,z(0)}(z) \leq \alpha\}.
\end{equation}
Similarly to the proof of \cite[Lemma 7.1]{cerrai2005largeinvariant}, we have the following lemma.

\begin{Lemma}\label{lem:zTzK0TnLalpha}
      Suppose that Assumptions \ref{ass:AQ}, \ref{ass:F} and \ref{ass:LFleqlambda1} hold. Then for any $\alpha \geq 0$ and $\delta > 0$, there exist $\bar{L} > 0$ and $\bar{T} > 0$ such that
      $$
            \big\{z(t): z \in K_{0,t}^{n,\bar{L}}(\alpha)\big\}
            \subset
            \Big\{y \in H_{n}:\rho^{H}(y,K^{n}(\alpha)) < \frac{\delta}{2}\Big\},
            \quad t \geq \bar{T},
      $$
      where $\rho^{H}(u,U) := \inf\limits_{u' \in U}|u-u'|, u \in H, U \subset H$.
\end{Lemma}

For any $K,L > 0$, $J \in \N^{+}$, we define
\begin{equation}
      H_{K,L,J}
      :=
      \{z \in C([0,J];H_{n}) : |z(0)| \leq K,
      |z(j)| > L, j = 1,2,\ldots,J \}.
\end{equation}
The proof of the following lemma is similar to that of Lemma 7.2 in \cite{cerrai2005largeinvariant}.

\begin{Lemma}\label{lem:infI0barJnyz310}
      Suppose that Assumptions \ref{ass:AQ}, \ref{ass:F} and \ref{ass:LFleqlambda1} hold. Then for any $\alpha,\delta > 0$ and $K > 0$, there exists $\bar{J} \in \N^{+}$ such that
      $$
            \inf\{I_{0,\bar{J}}^{n,z(0)}(z) : z \in H_{K,\bar{L},\bar{J}}\}
            >
            \alpha,
      $$
      where $\bar{L}$ is the constant introduced in Lemma \ref{lem:zTzK0TnLalpha} corresponding to $\alpha$ and $\delta$.
\end{Lemma}

Based on the above lemmas, we are in a position to show the LDP upper bound estimate.
\begin{Theorem}
\label{th:semiinvariantmeasureupperbounded}
      Suppose that Assumptions \ref{ass:AQ}, \ref{ass:F} and \ref{ass:LFleqlambda1} hold. Then for any $\alpha \geq 0$, $\delta > 0$ and $\gamma > 0$, there exists $\varepsilon_{0} > 0$ such that
      $$
            \mu^{\varepsilon,n}(\{y \in H_{n} : \rho^{H}(y,K^{n}(\alpha)) \geq \delta\})
            \leq
            \exp\Big(-\frac{\alpha+\gamma}{\varepsilon^{2}}\Big),
            \quad \varepsilon \leq \varepsilon_{0}.
      $$
\end{Theorem}

\begin{proof}
      By the invariance of $\mu^{\varepsilon,n}$ and \eqref{eq:K0TnLalpha}, we have that for any $t \geq 0$, $K,\, L > 0$, $J \in \N^{+}$,
      \begin{equation}\label{eq:muvarepsilonnyHn}
      \begin{split}
            &\mu^{\varepsilon,n}(\{y \in H_{n} :
            \rho^{H}(y,K^{n}(\alpha)) \geq \delta\})
            =
            \int_{H_{n}} \P(\rho^{H}(X_{y}^{\varepsilon,n}(t),K^{n}(\alpha))
             \geq \delta)
            \,\mu^{\varepsilon,n}(\dif{y})
            \\\leq&
            \mu^{\varepsilon,n}(\{y \in H_{n} : |y| > K\})
            +
            \int_{|y| \leq K}
                \P(\rho^{H}(X_{y}^{\varepsilon,n}(t),K^{n}(\alpha)) \geq \delta)
            \,\mu^{\varepsilon,n}(\dif{y})
            \\\leq&
            \mu^{\varepsilon,n}(\{y \in H_{n} : |y| > K\})
            +
            \sup_{|y| \leq K}
            \P(X_{y}^{\varepsilon,n} \in H_{K,L,J})
            \\&+
            \int_{|y| \leq K}
                \P(\rho^{H}(X_{y}^{\varepsilon,n}(t),K^{n}(\alpha)) \geq \delta,
                X_{y}^{\varepsilon,n} \notin H_{K,L,J})
            \,\mu^{\varepsilon,n}(\dif{y}).
      \end{split}
      \end{equation}
       For the last term in \eqref{eq:muvarepsilonnyHn}, applying
      $
            \{X_{y}^{\varepsilon,n} \in H_{K,L,J}\}
            =
            \bigcap\limits_{j=1}^{J}
            \{|X_{y}^{\varepsilon,n}(j)| > L\}
      $
      and the Markov property of $\{X_{y}^{\varepsilon,n}(t)\}_{t \geq 0}$ yields
      \begin{align*}\label{eq:intyHnleqK}
            &\int_{|y| \leq K}
                \P(\rho^{H}(X_{y}^{\varepsilon,n}(t),K^{n}(\alpha)) \geq \delta,
                X_{y}^{\varepsilon,n} \notin H_{K,L,J})
            \,\mu^{\varepsilon,n}(\dif{y})
            \\\leq&
            \sum\limits_{j=1}^{J}
            \int_{|y| \leq K}
                \P\big(\rho^{H}(X_{y}^{\varepsilon,n}(t),K^{n}(\alpha)) \geq \delta,
                |X_{y}^{\varepsilon,n}(j)| \leq L\big)
            \,\mu^{\varepsilon,n}(\dif{y})
            \\=&
            \sum\limits_{j=1}^{J}
            \int_{|y| \leq K}
                \P\big(\rho^{H}(X_{y}^{\varepsilon,n}(t-j),K^{n}(\alpha)) \geq \delta,
                |X_{y}^{\varepsilon,n}(j-j)| \leq L\big)
            \,\mu^{\varepsilon,n}(\dif{y})
            \\\leq&
            \sum\limits_{j=1}^{J}
            \sup_{|y| \leq L}
            \P\big(\rho^{H}(X_{y}^{\varepsilon,n}(t-j),K^{n}(\alpha)) \geq \delta\big),
            \quad t > J.
      \end{align*}
      Substituting the above estimate into \eqref{eq:muvarepsilonnyHn} leads to
      \begin{equation}
      \begin{split}
            \mu^{\varepsilon,n}&(\{y \in H_{n} :
            \rho^{H}(y,K^{n}(\alpha)) \geq \delta\})
            \leq
            \mu^{\varepsilon,n}(\{y \in H_{n} : |y| > K\})
            \\&+
            \sum\limits_{j=1}^{J}
            \sup_{|y| \leq L}
            \P\big(\rho^{H}(X_{y}^{\varepsilon,n}(t-j),K^{n}(\alpha)) \geq \delta\big)
            +
            \sup_{|y| \leq K}
            \P(X_{y}^{\varepsilon,n} \in H_{K,L,J}),
            \quad t > J.
      \end{split}
      \end{equation}
      It follows from Lemma \ref{lem:stochasticevolutiontailestimate} that for any $\alpha \geq 0$, there exists $\bar{K} > 0$ such that
      \begin{equation}\label{eq:muvnyHnK}
            \mu^{\varepsilon,n}(\{y \in H_{n} : |y| > \bar{K}\})
            \leq
            \exp\Big(-\frac{\alpha}{\varepsilon^{2}}\Big),
            \quad \varepsilon \leq 1.
      \end{equation}
      Lemma \ref{lem:zTzK0TnLalpha} shows that for any $\alpha > 0$ and $\delta > 0$, there exists $\bar{L} > 0$ and $\bar{T} > 0$ such that for any $t \geq \bar{T}$ and $y \in H_{n}$ with $|y| \leq \bar{L}$,
      \begin{equation*}
            \rho^{H}(z(t),K^{n}(\alpha)) < \frac{\delta}{2},
            \quad z \in K_{0,t}^{n,\bar{L}}(\alpha).
      \end{equation*}
      Then 
      the triangle inequality
      $
            \rho^{H}(X_{y}^{\varepsilon,n}(t),K^{n}(\alpha))
            \leq
            \rho^{H}(X_{y}^{\varepsilon,n}(t),z(t))
            +
            \rho^{H}(z(t),K^{n}(\alpha))
      $
      implies $$\{\rho^{H}(X_{y}^{\varepsilon,n}(t),K^{n}(\alpha)) \geq \delta\} \subset
      \{\rho^{C([0,t];H)}(X_{y}^{\varepsilon,n},z)
      \geq \frac{\delta}{2}\},
      \quad z \in K_{0,t}^{n,\bar{L}}(\alpha).$$
      By the arbitrariness of $z \in K_{0,t}^{n,\bar{L}}(\alpha)$, we get
      $$\{\rho^{H}(X_{y}^{\varepsilon,n}(t),K^{n}(\alpha)) \geq \delta\} \subset
      \{\rho^{C([0,t];H)}
            (X_{y}^{\varepsilon,n},K_{0,t}^{n,y}(\alpha))
      \geq \frac{\delta}{2}\},
      \quad y \in H_{n}, |y| \leq \bar{L}.$$
      This together with \eqref{eq:semisoluldpupper} implies that for any $t \geq \bar{T}$ there exists $\varepsilon(t) > 0$ such that for any $y \in H_{n}$ with $|y| \leq \bar{L}$,
      \begin{align*}
            \P\big(\rho^{H}(X_{y}^{\varepsilon,n}(t),K^{n}(\alpha)) \geq \delta\big)
            \leq
            \P\big(\rho^{C([0,t];H)}(X_{y}^{\varepsilon,n},
                   K_{0,t}^{n,y}(\alpha)) \geq \delta/2\big)
            \leq
            \exp\Big(-\frac{\alpha-\gamma/2}{\varepsilon^{2}}\Big),
            \quad \varepsilon \leq \varepsilon(t).
      \end{align*}
      Taking $t = \bar{T} + J$ and $\varepsilon_{1} := \min\{\varepsilon(t-j), j = 1,\ldots,J\}$, we have
      $$
            \sum\limits_{j=1}^{J}
            \sup_{|y| \leq \bar{L}}
            \P\big(\rho^{H}(X_{y}^{\varepsilon,n}(t-j),K^{n}(\alpha)) \geq \delta\big)
            \leq
            J \exp\Big(-\frac{\alpha-\gamma/2}{\varepsilon^{2}}\Big),
            \quad \varepsilon \leq \varepsilon_{1}.
      $$
      Moreover, Lemma \ref{lem:infI0barJnyz310} implies that there exists $\bar{J} \in \N^{+}$ such that $\rho^{C([0,\bar{J}];H)}(H_{\bar{K},\bar{L},\bar{J}},
            K_{0,\bar{J}}^{n,y}(\alpha)) > 0$. It follows from \eqref{eq:semisoluldpupper} that there exists $\varepsilon_{2} > 0$ such that
      \begin{equation*}
      \begin{split}
            \sup_{|y| \leq \bar{K}}
            \P(X_{y}^{\varepsilon,n} \in H_{\bar{K},\bar{L},\bar{J}})
            \leq&
            \sup_{|y| \leq \bar{K}}
            \P\big(\rho^{C([0,\bar{J}];H)}
            (X_{y}^{\varepsilon,n},K_{0,\bar{J}}^{n,y}(\alpha)) \geq
            \rho^{C([0,\bar{J}];H)}(H_{\bar{K},\bar{L},\bar{J}},
            K_{0,\bar{J}}^{n,y}(\alpha))\big)
            \\\leq&
            \exp\Big(-\frac{\alpha-\gamma/2}{\varepsilon^{2}}\Big),
            \quad \varepsilon \leq \varepsilon_{2}.
      \end{split}
      \end{equation*}
      Combining the above estimates shows that for any $\varepsilon \leq \varepsilon_{3} := \min\{1,\varepsilon_{1},\varepsilon_{2}\}$,
      $$
            \mu^{\varepsilon,n}(\{y \in H_{n} :
            \rho^{H}(y,K^{n}(\alpha)) \geq \delta\})
            \leq
            \exp\Big(-\frac{\alpha}{\varepsilon^{2}}\Big)
            +
            (1+\bar{J})\exp\Big(-\frac{\alpha-\gamma/2}{\varepsilon^{2}}\Big).
      $$
      By taking sufficiently small $\varepsilon_{0} \leq \varepsilon_{3}$, we complete the proof.
\end{proof}

\subsection{Weakly asymptotical preservation for the LDP of $\{\mu^{\varepsilon,n}\}_{\varepsilon > 0}$}
\label{sebsec:semiinvameaLDPasy}

To estimate the error between the rate functions $V^{n}$ and $V$, we denote the effective domain of $V$ by
\begin{equation}
      \mathcal{D}_{V}
      :=
      \{u \in H : V(u) < \infty\}.
\end{equation}
Now we give the following definition of weakly asymptotical preservation  for the LDP of invariant measures by a numerical method.

\begin{Definition}
\label{def:semiinvaweakasympres}
      We say that 
      the semi-discrete numerical method \eqref{eq:SPDEsemi}
      weakly asymptotically preserves the LDP of $\{\mu^{\varepsilon}\}_{\varepsilon > 0}$ if for any $\kappa > 0$ and $u \in \mathcal{D}_{V}$, there exist $n \in \N^{+}$ and $u_{n} \in H_{n}$ such that
      $$
            |u-u_{n}| < \kappa,
            \qquad
            |V(u)-V^{n}(u_{n})| < \kappa.
      $$
\end{Definition}

\begin{Theorem}\label{th:semiinvameaLDPasympre}
      Suppose that Assumptions \ref{ass:AQ}, \ref{ass:F} and \ref{ass:LFleqlambda1} hold. If
      $F \equiv \textbf{0}$,
      then 
      the semi-discrete numerical method \eqref{eq:SPDEsemi}
      weakly asymptotically preserves the LDP of $\{\mu^{\varepsilon}\}_{\varepsilon > 0}$.
\end{Theorem}

\begin{proof}
      For any $u \in \mathcal{D}_{V}$, we set $u_{n} := P_{n}u, n \in \N^{+}$ and obtain $\lim\limits_{n \to \infty}|u-u_{n}| = 0$,
      i.e., for each $\kappa > 0$ there exists $n_{1} \in \N^{+}$ such that
      \begin{equation}\label{eq:xynHkappa}
            |u-u_{n}| < \kappa,\quad n \geq n_{1}.
      \end{equation}
      The definition of $V(u)$ implies that for the above given $\kappa > 0$ there exists $T_{\kappa} > 0$ and $z_{\kappa} \in C([0,T_{\kappa}];H)$ with $z_{\kappa}(0) = 0, z_{\kappa}(T_{\kappa}) = u$ such that
      $$
            I_{0,T_{\kappa}}^{0}(z_{\kappa})
            \leq
            V(u) + \frac{\kappa}{2} < +\infty,
      $$
      which in combination with \eqref{eq:I0Txz} yields $z_{\kappa} \in W_{2,Q}^{1,2}(T_{\kappa})$. Define
      \begin{equation*}
            \varphi_{\kappa}(t)
            =
            Q^{-\frac{1}{2}}
            \Big(\frac{\diff{z_{\kappa}(t)}}{\diff{t}}
            - Az_{\kappa}(t) - F(z_{\kappa}(t))\Big)
      \end{equation*}
      for almost all $t \in [0,T_{\kappa}]$, then $\varphi_{\kappa} \in L^{2}(0,T_{\kappa};H), z_{0,0}^{\varphi_{\kappa}} = z_{\kappa}$ and $\frac{1}{2}|\varphi_{\kappa}|_{L^{2}(0,T_{\kappa};H)}^{2} = I_{0,T_{\kappa}}^{0}(z_{\kappa}) < +\infty$, which together with Lemma \ref{lem:spatialregularity} gives
      \begin{equation}\label{eq:zkappatdotH2}
            |z_{\kappa}(t)|_{\dot{H}^{2}}
            =
            |z_{0,0}^{\varphi_{\kappa}}(t)|_{\dot{H}^{2}}
            \leq
            C,
            \quad t \in [0,T_{\kappa}].
      \end{equation}
      As $P_{n}z_{\kappa} \in \mathcal{H}_{1}^{n,0}(T_{\kappa}),n \in \N^{+}$, we assert that
      \begin{equation}\label{eq:I0Tkappan0Pnzkappa342}
            \lim\limits_{n \to \infty}
            I_{0,T_{\kappa}}^{n,0}(P_{n}z_{\kappa})
            =
            I_{0,T_{\kappa}}^{0}(z_{\kappa}),
      \end{equation}
      which shows that there exists $n_{2} \in \N^{+}$ such that
      \begin{equation}\label{eq:I0Tkappan0Pnzkappa343}
            |I_{0,T_{\kappa}}^{0}(z_{\kappa})
            -I_{0,T_{\kappa}}^{n,0}(P_{n}z_{\kappa})|
            \leq
            \frac{\kappa}{2},
            \quad n \geq n_{2}.
      \end{equation}
      In fact, applying the identity
      $|u|^{2} - |v|^{2} + |u-v|^{2} = 2\langle u,u-v \rangle$ for all $u,v \in H$, \eqref{eq:I0Txz}, \eqref{eq:I0Tnyz}
      and the H\"{o}lder inequality leads to
      \begin{align*}
            &~|I_{0,T_{\kappa}}^{0}(z_{\kappa})
            -I_{0,T_{\kappa}}^{n,0}(P_{n}z_{\kappa})|
            \leq
            \frac{1}{2}
            \int_{0}^{T_{\kappa}}
            \Bigg|
                \Big|
                Q^{-\frac{1}{2}}
                \Big(\frac{\diff{z_{\kappa}(t)}}{\diff{t}}
                - Az_{\kappa}(t) - F(z_{\kappa}(t))\Big)
                \Big|^{2}
                \\&~-
                \Big|
                Q_{n}^{-\frac12}
                \Big(\frac{\diff{P_{n}z_{\kappa}(t)}}{\diff{t}}
                - A_{n}P_{n}z_{\kappa}(t) - F_{n}(P_{n}z_{\kappa}(t))\Big)
                \Big|^{2}
            \Bigg|
            \diff{t}
            \\=&~
            \frac{1}{2}
            \int_{0}^{T_{\kappa}}
            \Bigg|
            -\Big|Q^{-\frac{1}{2}}
            \Big(\frac{\diff{z_{\kappa}(t)}}{\diff{t}}
            - Az_{\kappa}(t) - F(z_{\kappa}(t))\Big)
            -
            Q_{n}^{-\frac12}
            \Big(\frac{\diff{P_{n}z_{\kappa}(t)}}{\diff{t}}
            - AP_{n}z_{\kappa}(t) - F_{n}(P_{n}z_{\kappa}(t))\Big)
            \Big|^{2}
            \\&~+
            2\Big\langle
            Q^{-\frac{1}{2}}
            \Big(\frac{\diff{z_{\kappa}(t)}}{\diff{t}}
            - Az_{\kappa}(t) - F(z_{\kappa}(t))\Big)
            -
            Q_{n}^{-\frac12}
            \Big(\frac{\diff{P_{n}z_{\kappa}(t)}}{\diff{t}}
            - AP_{n}z_{\kappa}(t) - F_{n}(P_{n}z_{\kappa}(t))\Big),
            \\&~~~~~~~~~~Q^{-\frac{1}{2}}
            \Big(\frac{\diff{z_{\kappa}(t)}}{\diff{t}}
            - Az_{\kappa}(t) - F(z_{\kappa}(t))\Big)
            \Big\rangle
            \Bigg|\diff{t}
            \\\leq&~
            \frac{1}{2}\int_{0}^{T_{\kappa}} \Theta_{n}(t) \diff{t}
            +
            \Big(2I_{0,T_{\kappa}}^{0}(z_{\kappa})
            \int_{0}^{T_{\kappa}}\Theta_{n}(t)\diff{t}\Big)^{\frac{1}{2}},
      \end{align*}
      where
      \begin{equation*}
            \Theta_{n}(t)
            :=
            \Big|Q^{-\frac{1}{2}}
            \Big(\frac{\diff{z_{\kappa}(t)}}{\diff{t}}
            - Az_{\kappa}(t) - F(z_{\kappa}(t))\Big)
            -
            Q_{n}^{-\frac12}
            \Big(\frac{\diff{P_{n}z_{\kappa}(t)}}{\diff{t}}
            - AP_{n}z_{\kappa}(t) - F_{n}(P_{n}z_{\kappa}(t))\Big)
            \Big|^{2}.
      \end{equation*}
      Observing $Q^{-\frac12}u = Q_{n}^{-\frac12}u,u \in H_{n}$, we apply the mean value formula to get
      \begin{align*}
            \Theta_{n}(t)
            =&~
            \Big|(I-P_{n})Q^{-\frac{1}{2}}
            \Big(\frac{\diff{z_{\kappa}(t)}}{\diff{t}}
            - Az_{\kappa}(t) - F(z_{\kappa}(t))\Big)
            -
            Q^{-\frac{1}{2}}
            \big(F_{n}(z_{\kappa}(t)) - F_{n}(P_{n}z_{\kappa}(t))\big)
            \Big|^{2}
            \\=&~
            \Big|(I-P_{n})Q^{-\frac{1}{2}}
            \Big(\frac{\diff{z_{\kappa}(t)}}{\diff{t}}
            - Az_{\kappa}(t) - F(z_{\kappa}(t))\Big)
            \\&~-
            \int_{0}^{1}Q^{-\frac{1}{2}}
            F_{n}^{'}\big(P_{n}z_{\kappa}(t)
            +r(z_{\kappa}(t)-P_{n}z_{\kappa}(t))\big)
            \big(z_{\kappa}(t)-P_{n}z_{\kappa}(t)\big)\diff{r}\Big|^{2}.
      \end{align*}
      On one hand, we use \eqref{eq:hybridAQ}, \eqref{eq:auxiliaryassumption}, $F_{n}^{'} = P_{n}F'$ and $|I-P_{n}|_{\mathcal{L}(H)} = 1, n \in \N^{+}$ to show
      \begin{equation*}
      \begin{split}
            \Theta_{n}(t)
            \leq
            2\Big|Q^{-\frac{1}{2}}
            \Big(\frac{\diff{z_{\kappa}(t)}}{\diff{t}}
            - Az_{\kappa}(t) - F(z_{\kappa}(t))\Big)\Big|^{2}
            +
            2L^{2}\big(1+4|z_{\kappa}(t)|_{\dot{H}^{2}}^{2}\big)^{2}
            |z_{\kappa}(t)|_{\dot{H}^{2}}^{2}
      \end{split}
      \end{equation*}
      with
      \begin{equation*}
            \int_{0}^{T_{\kappa}} \Big|Q^{-\frac{1}{2}}
            \Big(\frac{\diff{z_{\kappa}(t)}}{\diff{t}} - Az_{\kappa}(t)
            - F(z_{\kappa}(t))\Big)\Big|^{2}  \diff{t}
            =
            2I_{0,T_{\kappa}}^{0}(z_{\kappa})
            <
            \infty,
      \end{equation*}
      \begin{equation*}
            \int_{0}^{T_{\kappa}}
            \big(1+4|z_{\kappa}(t)|_{\dot{H}^{2}}^{2}\big)^{2}
            |z_{\kappa}(t)|_{\dot{H}^{2}}^{2}
            \diff{t}
            \leq
            T_{\kappa}C^{2}\big(1+4C^{2}\big)^{2}
            <
            \infty
      \end{equation*}
      due to \eqref{eq:zkappatdotH2}. On the other hand, using $\lim\limits_{n \to \infty} P_{n}u = u, u \in H$ and
      \eqref{eq:auxiliaryassumption} yields
      \begin{equation*}
      \begin{split}
            \lim\limits_{n \to \infty} \Theta_{n}(t)
            \leq&
            2\lim\limits_{n \to \infty}
            \Big|(I-P_{n})Q^{-\frac{1}{2}}
            \Big(\frac{\diff{z_{\kappa}(t)}}{\diff{t}}
            - Az_{\kappa}(t) - F(z_{\kappa}(t))\Big)\Big|^{2}
            \\&+
            2L^{2}\big(1+4|(-A)z_{\kappa}(t)|^{2}\big)^{2}
            \lim\limits_{n \to \infty}
            |(I-P_{n})(-A)z_{\kappa}(t)|^{2} = 0.
      \end{split}
      \end{equation*}
      Combining the above results and using the dominated convergence theorem ensure \eqref{eq:I0Tkappan0Pnzkappa342}.
      As $P_{n}z_{\kappa} \in \mathcal{H}_{1}^{n,0}(T_{\kappa}) \subset C([0,T_{\kappa}];H_{n})$ satisfying $P_{n}z_{\kappa}(0) = 0$ and $P_{n}z_{\kappa}(T_{\kappa}) = P_{n}u = u_{n}$ for each $n \in \N^{+}$, the definition of $V^{n}(u_{n})$ and \eqref{eq:I0Tkappan0Pnzkappa343} lead to
      \begin{equation}\label{eq:345}
            V^{n}(u_{n})
            \leq
            I_{0,T_{\kappa}}^{n,0}(P_{n}z_{\kappa})
            \leq
            I_{0,T_{\kappa}}^{0}(z_{\kappa}) + \frac{\kappa}{2}
            \leq
            V(u) + \kappa,
            \quad n \geq \max\{n_{1},n_{2}\}.
      \end{equation}
      It remains to show $V(u) \leq V^{n}(u_{n}) + \kappa$ for sufficiently large $n \in \N^{+}$.
      Now for each $n \geq \max\{n_{1},n_{2}\}$, the definition of $V^{n}(u_{n})$ implies that for the above given $\kappa > 0$ there exists $T_{n,\kappa} > 0$, $z_{n,\kappa} \in C([0,T_{n,\kappa}];H_{n})$, $z_{n,\kappa}(0) = 0$, $z_{n,\kappa}(T_{n,\kappa}) = u_{n}$ such that
      \begin{equation}\label{eq:346}
            I_{0,T_{n,\kappa}}^{n,0}(z_{n,\kappa})
            \leq
            V^{n}(u_{n}) + \frac{\kappa}{2}
            \leq
            V(u) + \frac{3}{2}\kappa
            <
            + \infty,
      \end{equation}
      which in combination with \eqref{eq:I0Tnyz} yields $z_{n,\kappa} \in \mathcal{H}_{1}^{n,0}(T_{n,\kappa})$.
      Since 
      $z_{n,\kappa} \in W_{2,Q}^{1,2}(T_{n,\kappa})$, the definition of $V(u_{n})$ in \eqref{eq:Vny} shows
      \begin{equation}\label{eq:347}
            V(u_{n})
            \leq
            I_{0,T_{n,\kappa}}^{0}(z_{n,\kappa}).
      \end{equation}
      Due to 
      $F \equiv \textbf{0}$,
      we have
      \begin{equation}\label{eq:348}
      \begin{split}
            I_{0,T_{n,\kappa}}^{n,0}(z_{n,\kappa})
            =&
              \frac{1}{2}
              \int_{0}^{T_{n,\kappa}}
                  \Big|Q_{n}^{-\frac12}
                  \Big(\frac{\diff{z_{n,\kappa}(t)}}{\diff{t}}
                       -
                       A_{n}z_{n,\kappa}(t)\Big)
                  \Big|^{2}
              \diff{t}
            \\=&
              \frac{1}{2}
              \int_{0}^{T_{n,\kappa}}
                  \Big|Q^{-\frac12}
                  \Big(\frac{\diff{z_{n,\kappa}(t)}}{\diff{t}}
                       -
                       Az_{n,\kappa}(t)\Big)
                  \Big|^{2}
              \diff{t}
              =
              I_{0,T_{n,\kappa}}^{0}(z_{n,\kappa}).
      \end{split}
      \end{equation}
      As $\lim\limits_{n \to \infty} u_{n} = u$, the semi-lower continuity of $V$ implies $V(u) \leq \liminf\limits_{n \to \infty} V(u_{n})$, which means that for the above given $\kappa > 0$ there exists $n_{3} \in \N^{+}$ such that
      \begin{equation}\label{eq:349}
            V(u) \leq \liminf_{n \to \infty} V(u_{n})
            \leq
            V(u_{n}) + \frac{\kappa}{2},
            \quad n \geq n_{3}.
      \end{equation}
      Combining \eqref{eq:346}--\eqref{eq:349} leads to
      \begin{equation}\label{eq:350}
            V(u)
            \leq
            V^{n}(u_{n}) + \kappa,
            \quad n \geq \max\{n_{1},n_{2},n_{3}\}.
      \end{equation}
      Finally, \eqref{eq:xynHkappa}, \eqref{eq:345} and \eqref{eq:350} finish the proof.
\end{proof}

\section{Spatio-temporal discretization and its LDPs}
\label{sec:temporal}
In Subsection \ref{eq:AEEM}, we obtain a full-discrete numerical approximation $\{Y_{y,m}^{\varepsilon,n}\}_{m \in \N}$ by applying the accelerated exponential Euler method to \eqref{eq:SPDEsemi} and give the uniform LDP of sample paths of $\{Y_{y}^{\varepsilon, n}\}_{\varepsilon > 0}$, which is a continuous version of $\{Y_{y,m}^{\varepsilon,n}\}_{m \in \N}$.
Then the weakly asymptotical preservation for the LDP of $\{X_{x}^{\varepsilon}\}_{\varepsilon > 0}$ is given in Theorem \ref{th:fullLDPasympre}.
Subsection \ref{subsec:fullinvameaLDP} shows the LDP of invariant measures $\{\mu^{\varepsilon,n,\tau}\}_{\varepsilon > 0}$ of  $\{Y_{y}^{\varepsilon, n}\}_{\varepsilon > 0}$ and establishes its weakly asymptotical preservation for the LDP of $\{\mu^{\varepsilon}\}_{\varepsilon > 0}$.
\subsection{Accelerated exponential Euler method}
\label{eq:AEEM}

Let $\tau > 0$ be a uniform time stepsize and $t_{m} := m\tau, m \in \N$ the grid points, the accelerated exponential Euler method (see \cite{jentzen2009overcoming}) applied to \eqref{eq:SPDEsemi} is given
by setting $Y_{y,0}^{\varepsilon,n} = y$ and
\begin{equation}\label{eq:EE}
      Y_{y,m+1}^{\varepsilon,n}
      =
      E_{n}(\tau)Y_{y,m}^{\varepsilon,n}
      +
      \int_{t_{m}}^{t_{m+1}}
      E_{n}(t_{m+1}-s)F_{n}(Y_{y,m}^{\varepsilon,n}) \diff{s}
      +
      \varepsilon \int_{t_{m}}^{t_{m+1}}  E_{n}(t_{m+1}-s) Q_{n}^{\frac12}\diff{W_{n}(s)}
\end{equation}
for each $m \in \N$. To give a continuous time approximation of $\{Y_{y,m}^{\varepsilon,n}\}_{m \in \N}$, we define $\{Y_{y}^{\varepsilon,n}(t)\}_{t \geq 0}$ by
\begin{equation}\label{eq:4242}
      Y_{y}^{\varepsilon,n}(t)
      =
      E_{n}(t-t_{m}) Y_{y,m}^{\varepsilon,n}
      +
      \int_{t_{m}}^{t} E_{n}(t-s)
      F_{n}(Y_{y,m}^{\varepsilon,n}) \diff{s}
      +
      \varepsilon \int_{t_{m}}^{t} E_{n}(t-s) Q_{n}^{\frac12}\diff{W_{n}(s)}
\end{equation}
for all $t \in [t_{m},t_{m+1}), m \in \N$.
Let $\lfloor \cdot \rfloor$ represent the largest integer no larger than the corresponding constant.
As $Y_{y}^{\varepsilon,n}(t_{m}) = Y_{y,m}^{\varepsilon,n},m \in \N$, i.e., $Y_{y}^{\varepsilon,n}(\tau \lfloor t/\tau \rfloor) = Y_{y,\lfloor t/\tau \rfloor}^{\varepsilon,n}, t \geq 0$, we have
\begin{equation}\label{eq:saptiotemporalconteq}
\begin{split}
      &\diff{Y_{y}^{\varepsilon,n}(t)}
      =
      \big(A_{n}Y_{y}^{\varepsilon,n}(t)
      +
      F_{n}(Y_{y}^{\varepsilon,n}(\tau \lfloor t/\tau \rfloor)) \big)\diff{t}
      +
      \varepsilon Q_{n}^{\frac12}\diff{W_{n}(t)},
      \quad t > 0,
      \quad Y_{y}^{\varepsilon,n}(0) = y.
\end{split}
\end{equation}
For any $\phi \in L^{2}(0,T;H_{n})$, it follows from \cite[Theorem 1.1]{peszat1994large} that the skeleton equation
\begin{equation*}
      \frac{\diff{z_{0,y}^{n,\tau,\phi}(t)}}{\diff{t}}
      =
      A_{n}z_{0,y}^{n,\tau,\phi}(t)
      +
      F_{n}(z_{0,y}^{n,\tau,\phi}(\tau \lfloor t/\tau \rfloor))
      +
      Q_{n}^{\frac12}\phi(t),
      \quad t > 0,
      \quad z_{0,y}^{n,\tau,\phi}(0) = y
\end{equation*}
admits a unique mild solution $\{z_{0,y}^{n,\tau,\phi}(t)\}_{t \geq 0}$. 
We will prove that $\{Y_{y}^{\varepsilon, n}\}_{\varepsilon > 0}$ satisfies a  uniform LDP on $C([0,T];H_{n})$ with the rate function
$I_{0,T}^{n,\tau,y} \colon C([0,T];H_{n}) \to [0,+\infty]$
defined by
\begin{equation}\label{eq:SPDEfullratefun}
      I_{0,T}^{n,\tau,y}(z)
      =
      \frac{1}{2} \inf\{|\phi|_{L^{2}(0,T;H)}^{2}:
      \phi \in L^{2}(0,T;H_{n}), z_{0,y}^{n,\tau,\phi}= z\},
      \quad z \in C([0,T];H_{n}).
\end{equation}
By the arguments for \eqref{eq:I0Tnyz}, we have
\begin{equation}\label{eq:I0TnMyz}
      I_{0,T}^{n,\tau,y}(z)
      =
      \begin{cases}
        \frac{1}{2}
        \int_{0}^{T}
            \Big|Q_{n}^{-\frac12}\Big(\frac{\diff{z(t)}}{\diff{t}}
                 -
                 A_{n}z(t) - F_{n}(z(\tau \lfloor t/\tau \rfloor))\Big)
            \Big|^{2}
        \diff{t}
        , & \mbox{if } z \in \mathcal{H}_{1}^{n,y}(T), \\
        +\infty, & \mbox{otherwise}.
      \end{cases}
\end{equation}

Similarly to the proof of Theorem \ref{thm:semisolutionLDP}, we establish the uniform LDP of sample paths of
$\{Y_{y}^{\varepsilon,n}\}_{\varepsilon > 0}$.

\begin{Theorem}
\label{th:fullLDPsamplepath}
      Suppose that Assumptions \ref{ass:AQ} and \ref{ass:F} hold. Then $\{Y_{y}^{\varepsilon,n}\}_{\varepsilon > 0}$ satisfies a uniform LDP on $C([0,T];H_{n})$ with the rate $\frac{1}{\varepsilon^{2}}$ and the good rate function $I_{0,T}^{n,\tau,y}$ given by \eqref{eq:SPDEfullratefun}, i.e.,
      \begin{enumerate}
        \item [(i)] compact level set: for any $T > 0$, $n \in \N^{+}$ and $y \in H_{n}$, the level set $K_{0,T}^{n,\tau,y}(\alpha) :=
            \{z \in C([0,T];H_{n}) : I_{0,T}^{n,\tau,y}(z) \leq \alpha\}$
            is compact for every $\alpha \geq 0$;

        \item [(ii)] uniform lower bound: for any $T > 0$, $n \in \N^{+}$, $\alpha \geq 0$, $\delta > 0$, $\gamma > 0$ and $K > 0$, there exists $\varepsilon_{0} > 0$ such that for any $y \in H_{n}$ with $|y| \leq K$ and $z \in K_{0,T}^{n,\tau,y}(\alpha)$,
            \begin{equation}\label{eq:fullsoluldplower}
            \P\big(|Y_{y}^{\varepsilon,n}-z|_{C([0,T];H)} < \delta\big)
            \geq
            \exp\Big(-\frac{I_{0,T}^{n,\tau,y}(z)+\gamma}{\varepsilon^2}\Big),
            \quad \varepsilon \leq \varepsilon_{0};
            \end{equation}
        \item [(iii)] uniform upper bound: for any $T > 0$, $n \in \N^{+}$, $\alpha \geq 0$, $\delta > 0$, $\gamma > 0$ and $K > 0$, there exists $\varepsilon_0 > 0$ such that for any $y \in H_{n}$ with $|y| \leq K$,
            \begin{equation}
            \P\big(\rho^{C([0,T];H)}(Y_{y}^{\varepsilon,n},
            K_{0,T}^{n,\tau,y}(\alpha)) \geq \delta\big)
            \leq
            \exp\Big(-\frac{\alpha-\gamma}{\varepsilon^2}\Big),
            \quad \varepsilon \leq \varepsilon_{0}.
            \end{equation}
      \end{enumerate}
\end{Theorem}

To characterize the error between the rate functions $I_{0,T}^{n,\tau,y}$ and $I_{0,T}^{x}$, we give the definition of weakly asymptotical preservation for the LDP of sample paths by a numerical method.
\begin{Definition}
      We say that 
      the fully-discrete numerical method \eqref{eq:saptiotemporalconteq}
      weakly asymptotically preserves the LDP of $\{X_{x}^{\varepsilon}\}_{\varepsilon > 0}$ if for any $\kappa > 0$ and $z \in W_{2,Q}^{1,2}(T)$ with $z(0) = x$, there exist $n \in \N^{+}$, $\tau > 0$ and $z_{n,\tau} \in \mathcal{H}_{1}^{n,y}(T)$ such that
      $$
            |z-z_{n,\tau}|_{C([0,T];H)} < \kappa,
            \quad\quad
            |I_{0,T}^{x}(z)-I_{0,T}^{n,\tau,y}(z_{n,\tau})| < \kappa.
      $$
\end{Definition}

Now we show that $\{Y_{y}^{\varepsilon,n}\}_{\varepsilon > 0}$ weakly asymptotically preserves the LDP of $\{X_{x}^{\varepsilon}\}_{\varepsilon > 0}$.
\begin{Theorem}
\label{th:fullLDPasympre}
      Suppose that Assumptions \ref{ass:AQ} and \ref{ass:F} hold. Then 
      the fully-discrete numerical method \eqref{eq:saptiotemporalconteq}
      weakly asymptotically preserves the LDP of $\{X_{x}^{\varepsilon}\}_{\varepsilon > 0}$.
\end{Theorem}

\begin{proof}
      For any $\kappa > 0$ and $z \in W_{2,Q}^{1,2}(T)$ with $z(0) = x$, it follows from Theorem \ref{thm:spatialsolutionwap} that there exist $n \in \N^{+}$ and $z_{n} \in \mathcal{H}_{1}^{n,y}(T)$ such that
      \begin{equation}\label{eq:part111}
            |z-z_{n}|_{C([0,T];H)} < \frac{\kappa}{2} < \kappa,
            \quad\quad
            |I_{0,T}^{x}(z)-I_{0,T}^{n,y}(z_{n})| < \frac{\kappa}{2}.
      \end{equation}
      Then we use \eqref{eq:I0Tnyz}, \eqref{eq:I0TnMyz}, the identity
      $|u|^{2} - |v|^{2} + |u-v|^{2} = 2\langle u,u-v \rangle, u,v \in H$ and the H\"{o}lder inequality to get
      \begin{align*}
            |I_{0,T}^{n,y}(z_{n})
            -I_{0,T}^{n,\tau,y}(z_{n})|
            \leq&~
            \frac{1}{2}
                \int_{0}^{T}
                \Bigg|
                \Big|Q_{n}^{-\frac12}
                \Big(\frac{\diff{z_{n}(t)}}{\diff{t}}
                    -
                    A_{n}z_{n}(t)
                    -
                    F_{n}(z_{n}(t))\Big)\Big|^{2}
                \\&~-
                \Big|Q_{n}^{-\frac12}
                \Big(\frac{\diff{z_{n}(t)}}{\diff{t}}
                    -
                    A_{n}z_{n}(t)
                    -
                    F_{n}(z_{n}(\tau \lfloor t/\tau \rfloor))\Big)
                    \Big|^{2}
                \Bigg|
                \diff{t}
            \\=&~
            \frac{1}{2}
                \int_{0}^{T}
                \Big|
                2\Big\langle
                Q_{n}^{-\frac12}
                \big(F_{n}(z_{n}(\tau \lfloor t/\tau \rfloor))
                - F_{n}(z_{n}(t))\big),
                \\&~Q_{n}^{-\frac12}\Big(\frac{\diff{z_{n}(t)}}{\diff{t}}
                -
                A_{n}z_{n}(t) - F_{n}(z_{n}(t))\Big)\Big\rangle
                \\&~-
                \big|Q_{n}^{-\frac12}
                \big(F_{n}(z_{n}(\tau \lfloor t/\tau \rfloor))
                -F_{n}(z_{n}(t))\big)\big|^{2}
                \Big|
            \diff{t}
            \\\leq&~
            \Big(2I_{0,T}^{n,y}(z_{n})\int_{0}^{T}
                \big|Q_{n}^{-\frac12}
                \big(F_{n}(z_{n}(\tau \lfloor t/\tau \rfloor))
                  -
                F_{n}(z_{n}(t))\big)
                \big|^{2}\diff{t}\Big)^{\frac12}
                \\&~+
                \frac{1}{2}\int_{0}^{T}
                \big|Q_{n}^{-\frac12}
                \big(F_{n}(z_{n}(\tau \lfloor t/\tau \rfloor))
                -
                F_{n}(z_{n}(t))\big)\big|^{2}\diff{t}.
      \end{align*}
      Applying \eqref{eq:hybridAQ}, \eqref{eq:auxiliaryassumption} and the mean value formula,
      we have
      \begin{align*}
            &~\big|Q_{n}^{-\frac12}
            \big(F_{n}(z_{n}(\tau \lfloor t/\tau \rfloor))
            -
            F_{n}(z_{n}(t))\big)\big|^{2}
            \\\leq&~
            \big|Q_{n}^{-\frac12}(-A_{n})^{-1}\big|_{\mathcal{L}(H_{n})}
            \int_{0}^{1}\big|(-A_{n})F_{n}^{'}
            \big(z_{n}(t)
            +r(z_{n}(\tau \lfloor t/\tau \rfloor)
            -z_{n}(t))\big)
            \big(z_{n}(\tau \lfloor t/\tau \rfloor)
            -z_{n}(t)\big)
            \big|^{2}\diff{r}
            \\\leq&~
            C\big(1+|z_{n}(t)|_{\dot{H}^{2}}^{2}
            +|z_{n}(\tau \lfloor t/\tau \rfloor)|_{\dot{H}^{2}}^{2}\big)^{2}
            \big|(-A_{n})\big(z_{n}(\tau \lfloor t/\tau \rfloor)
            -z_{n}(t)\big)\big|^{2},
      \end{align*}
      where $F_{n}^{'} = P_{n}F'$. As $z_{n} \in \mathcal{H}_{1}^{n,y}(T)$, we define
      $$
          \psi(t)
          :=
          Q_{n}^{-\frac12}
          \Big(\frac{\diff{z_{n}(t)}}{\diff{t}}
          -
          A_{n}z_{n}(t)
          -
          F_{n}(z_{n}(t))\Big)
      $$
      for almost all $t \in [0,T]$. Then $\psi \in L^{2}(0,T;H_{n})$, $z_{0,y}^{n,\psi} = z_{n}$ and $\frac{1}{2}|\psi|_{L^{2}(0,T;H)}^{2} = I_{0,T}^{n,y}(z_{n})$.
      Similarly to the proof of \eqref{eq:38}, we obtain
      $|z_{n}(t)|_{\dot{H}^{2}} =
      |z_{0,y}^{n,\psi}(t)|_{\dot{H}^{2}}
      \leq C(1+|y|_{\dot{H}^{2}}),t \in [0,T]$, and thus
      \begin{align*}
            &\big|Q_{n}^{-\frac12}
            \big(F_{n}(z_{n}(\tau \lfloor t/\tau \rfloor))
            -F_{n}(z_{n}(t))\big)\big|^{2}
            \leq
            C\big|(-A_{n})\big|_{\mathcal{L}(H_{n})}^{2}
            \big(1+|y|_{\dot{H}^{2}}^{2}\big)^{2}|y|_{\dot{H}^{2}}^{2}
            <
            +\infty.
      \end{align*}
      By $\lim\limits_{\tau \to 0}
      \big|z_{n}(\tau \lfloor t/\tau \rfloor)
      -z_{n}(t) \big| = 0, t \in [0,T]$, we have
      \begin{align*}
            \lim\limits_{\tau \to 0}
            \big|Q_{n}^{-\frac12}
            \big(F_{n}(z_{n}(\tau \lfloor t/\tau \rfloor))
            -F_{n}(z_{n}(t))\big)\big|^{2}
            \leq
            C\big|(-A_{n})\big|_{\mathcal{L}(H_{n})}^{2}
            \lim\limits_{\tau \to 0}
            \big|z_{n}(\tau \lfloor t/\tau \rfloor) - z_{n}(t)\big|^{2}
            =
            0.
      \end{align*}
      It follows from the bounded convergence theorem that
      \begin{equation}\label{eq:410410410}
      \lim\limits_{\tau \to 0}|I_{0,T}^{n,y}(z_{n})
            -I_{0,T}^{n,\tau,y}(z_{n})| = 0,
      \end{equation}
      and consequently that there exists $\tau > 0$ such that
      \begin{equation}\label{eq:part222}
            \big|I_{0,T}^{n,y}(z_{n})
            -I_{0,T}^{n,\tau,y}(z_{n})\big|
            \leq
            \frac{\kappa}{2}.
      \end{equation}
      For $n,\tau$ given by \eqref{eq:part111} and \eqref{eq:part222}, we set $z_{n,\tau} := z_{n}$ and use the triangle inequality to complete the proof.
\end{proof}

\subsection{LDP of invariant measures of spatio-temporal discretization}
\label{subsec:fullinvameaLDP}

In this subsection, we study the LDP of invariant measures of the spatio-temporal discretization. The following lemma gives the uniform boundedness of numerical solutions, which ensures the existence of the invariant measures of the full discretization. We refer the interested readers to \cite{chen2017approximation,hong2017numerical,
hong2019invariant,cui2021weak} for the investigation on the invariant measures of the full discretzations for SPDEs.

\begin{Lemma}\label{lem:uniformbound2ordermoment}
      Suppose that Assumptions \ref{ass:AQ}, \ref{ass:F} and \ref{ass:LFleqlambda1} hold. If $\tau \leq \tau_{0}$ with $\tau_{0}$ satisfying
      $\frac{e^{\lambda_{1}\tau_{0}}-1}{\lambda_{1}\tau_{0}} = \frac{\lambda_{1}+L_{F}}{2L_{F}}$, then there exists $C > 0$ independent of $t$ such that
      \begin{equation}\label{eq:Y0vartHn}
            \sup\limits_{t \geq 0}
            \E\big[|Y_{y}^{\varepsilon,n}(t)|^{2}\big]
            \leq
            C.
      \end{equation}
\end{Lemma}

\begin{proof}
      We rewrite \eqref{eq:EE} as
      \begin{align*}
            Y_{y,m}^{\varepsilon,n}
            =&
            E_{n}^{m}(\tau)Y_{y,0}^{\varepsilon,n}
            +
            \int_{0}^{t_{m}}E_{n}(t_{m}-s)
            F_{n}(Y_{y,\lfloor s/\tau \rfloor}^{\varepsilon,n})
            \diff{s}
            +
            \varepsilon \int_{0}^{t_{m}}  E_{n}(t_{m}-s) Q_{n}^{\frac12}\diff{W_{n}(s)}.
      \end{align*}
      Noting $\{\Gamma^{n}(t)\}_{t \geq 0}$ given by \eqref{eq:Gammant} and setting $\bar{Y}_{y,m}^{\varepsilon,n} := Y_{y,m}^{\varepsilon,n} - \varepsilon\Gamma^{n}(t_{m})$, it follows that
      \begin{align*}
            \bar{Y}_{y,m}^{\varepsilon,n}
            =&
            E_{n}^{m}(\tau)\bar{Y}_{y,0}^{\varepsilon,n}
            +
            \int_{0}^{t_{m}}E_{n}(t_{m}-s)
            F_{n}(\bar{Y}_{y,\lfloor s/\tau \rfloor}^{\varepsilon,n}
            +\varepsilon\Gamma^{n}(\tau \lfloor s/\tau \rfloor))
            \diff{s},
            \quad
            \bar{Y}_{y,0}^{\varepsilon,n} = y
      \end{align*}
      and thus
      \begin{align*}
            \bar{Y}_{y,m}^{\varepsilon,n}
            =&
            E_{n}(\tau)\bar{Y}_{y,m-1}^{\varepsilon,n}
            +
            (-A_{n})^{-1}(I-E_{n}(\tau))
            F_{n}(\bar{Y}_{y,m-1}^{\varepsilon,n}
            +\varepsilon\Gamma^{n}(t_{m-1})).
      \end{align*}
      Applying $|E_{n}(\tau)|_{\mathcal{L}(H_{n})} \leq e^{-\lambda_{1}\tau}$ and
      $|(-A_{n})^{-1}(I-E_{n}(\tau))|_{\mathcal{L}(H_{n})} \leq (1-e^{-\lambda_{1}\tau})/\lambda_{1} \leq \tau$ yields
      \begin{align*}
            |\bar{Y}_{y,m}^{\varepsilon,n}|_{L^{2}(\Omega;H)}
      \leq&
            e^{-\lambda_{1}\tau}
            |\bar{Y}_{y,m-1}^{\varepsilon,n}|_{L^{2}(\Omega;H)}
            +
            \tau \frac{1-e^{-\lambda_{1}\tau}}{\lambda_{1}\tau}
            |F_{n}(\bar{Y}_{y,m-1}^{\varepsilon,n}
            +\varepsilon\Gamma^{n}(t_{m-1}))|_{L^{2}(\Omega;H)}
      \\\leq&
            \Big(e^{-\lambda_{1}\tau}
            +
            \tau L_{F}\frac{1-e^{-\lambda_{1}\tau}}{\lambda_{1}\tau}\Big)
            |\bar{Y}_{y,m-1}^{\varepsilon,n}|_{L^{2}(\Omega;H)}
            +
            \tau L_{F}\frac{1-e^{-\lambda_{1}\tau}}{\lambda_{1}\tau}
            |\varepsilon\Gamma^{n}(t_{m-1})|_{L^{2}(\Omega;H)}.
      \end{align*}
      Recalling $\sup\limits_{t \geq 0}
      |\Gamma^{n}(t)|_{L^{2}(\Omega;H)} \leq C$ in \eqref{eq:GammanC0tHn2} and
      $\tau \leq \tau_{0}$ with $\tau_{0}$ satisfying
      $\frac{e^{\lambda_{1}\tau_{0}}-1}{\lambda_{1}\tau_{0}} = \frac{\lambda_{1}+L_{F}}{2L_{F}}$,
      we get
      \begin{align*}
            |\bar{Y}_{y,m}^{\varepsilon,n}|_{L^{2}(\Omega;H)}
      \leq&
            \Big(e^{-\lambda_{1}\tau}
            +
            \tau L_{F}e^{-\lambda_{1}\tau}
            +
            \tau\frac{\lambda_{1}-L_{F}}{2}e^{-\lambda_{1}\tau}\Big)
            |\bar{Y}_{y,m-1}^{\varepsilon,n}|_{L^{2}(\Omega;H)}
            +
            C\tau
      \\=&
            \Big(1 + \tau\frac{\lambda_{1}+L_{F}}{2}\Big)
            e^{-\lambda_{1}\tau}
            |\bar{Y}_{y,m-1}^{\varepsilon,n}|_{L^{2}(\Omega;H)}
            +
            C\tau
      \\\leq&
            e^{-\frac{\lambda_{1}-L_{F}}{2}\tau}
            |\bar{Y}_{y,m-1}^{\varepsilon,n}|_{L^{2}(\Omega;H)}
            +
            C\tau
      \\\leq&
            e^{-\frac{\lambda_{1}-L_{F}}{2}m\tau}|y|
            +
            C\tau\frac{e^{\frac{\lambda_{1}-L_{F}}{2}\tau}}
                  {e^{\frac{\lambda_{1}-L_{F}}{2}\tau}-1}
      \\\leq&
            |y|
            +
            2C\frac{e^{\frac{\lambda_{1}-L_{F}}{2}\tau_{0}}}
                  {\lambda_{1}-L_{F}}.
      \end{align*}
      Thus we assert that there exists $C > 0$ independent of $m \in \N$ such that $\sup\limits_{m \in \N}|Y_{y,m}^{\varepsilon,n}|_{L^{2}(\Omega;H)} \leq C$.
      It follows from \eqref{eq:4242} that
      \begin{align*}
            |Y_{y}^{\varepsilon,n}(t)|_{L^{2}(\Omega;H)}
      \leq&
            |Y_{y,m-1}^{\varepsilon,n}|_{L^{2}(\Omega;H)}
            +
            |F_{n}(Y_{y,m-1}^{\varepsilon,n})|_{L^{2}(\Omega;H)}
            +
            |\varepsilon \Gamma^{n}(t)|_{L^{2}(\Omega;H)}
            \leq
            C
      \end{align*}
      with $C$ being independent of $t$. Thus we complete the proof.
\end{proof}

By Lemma \ref{lem:uniformbound2ordermoment} and \cite[Proposition 7.10]{da2006introduction}, the family of probability measures $\{\mu_{t}^{\varepsilon,n,\tau}\}_{t > 0}$, defined by
      $$
            \mu_{t}^{\varepsilon,n,\tau}(B)
            :=
            \frac{1}{t} \int_{0}^{t}
            \P\big(Y_{0}^{\varepsilon,n}(s) \in B\big)
            \diff{s},
            \quad B \in \mathcal{B}(H_{n}), t > 0,
      $$
is tight. It follows from the Krylov--Bogoliubov theorem (\cite[Theorem 7.1]{da2006introduction}) that there exists $\{t_{i}\}_{i \in \N^{+}} \uparrow +\infty$ (possibly depending on $\varepsilon$) such that the sequence $\{\mu_{t_{i}}^{\varepsilon,n,\tau}\}_{i \in \N^{+}}$ converges weakly to some probability measure $\mu^{\varepsilon,n,\tau}$ on $(H_{n},\mathcal{B}(H_{n}))$, which is invariant for \eqref{eq:saptiotemporalconteq}. Following the procedures
in Subsection \ref{subsec:semiinvameaLDPs}, we establish the LDP for $\{\mu^{\varepsilon,n,\tau}\}_{\varepsilon > 0}$ on $H_{n}$ with the rate function $V^{n,\tau}$ defined by
\begin{equation}\label{eq:Vntauy}
      V^{n,\tau}(u)
      =
      \inf\{I_{0,T}^{n,\tau,0}(z) : T > 0, z \in C([0,T];H_{n}), z(0) = 0, z(T) = u \},
      \quad u \in H_{n}.
\end{equation}

\begin{Theorem}
      Under assumptions in Lemma \ref{lem:uniformbound2ordermoment}, $\{\mu^{\varepsilon,n,\tau}\}_{\varepsilon > 0}$ satisfies an LDP on $H_{n}$ with the rate $\frac{1}{\varepsilon^{2}}$ and the good rate function $V^{n,\tau}$ given by \eqref{eq:Vntauy}.
\end{Theorem}

We give the definition of weakly asymptotical preservation for the LDP of invariant measures by a numerical method to characterize the error between $V^{n,\tau}$ and $V$.

\begin{Definition}
      We say that 
      the fully-discrete numerical method \eqref{eq:saptiotemporalconteq}
      weakly asymptotically preserves the LDP of $\{\mu^{\varepsilon}\}_{\varepsilon > 0}$ if for any $\kappa > 0$ and $u \in \mathcal{D}_{V}$, there exist $n \in \N^{+}$, $\tau > 0$ and $u_{n,\tau} \in H_{n}$ such that
      $$
            |u-u_{n,\tau}| < \kappa,
            \qquad
            |V(u)-V^{n,\tau}(u_{n,\tau})| < \kappa.
      $$
\end{Definition}

\begin{Theorem}\label{fullinvameaLDPasypre}
      Suppose that assumptions in Lemma \ref{lem:uniformbound2ordermoment} hold.
      If 
      $F \equiv \textbf{0}$,
      then 
      the fully-discrete numerical method \eqref{eq:saptiotemporalconteq}
      weakly asymptotically preserves the LDP of $\{\mu^{\varepsilon}\}_{\varepsilon > 0}$.
\end{Theorem}

%
%
%
%

\begin{proof}
      For any $\kappa > 0$ and $u \in \mathcal{D}_{V}$, Theorem \ref{th:semiinvameaLDPasympre} shows that there exist
      $n \in \N^{+}$ and $u_{n} \in H_{n}$ such that
      \begin{equation}\label{eq:invapart111}
            |u-u_{n}| < \frac{\kappa}{2} < \kappa,
            \qquad
            |V(u)-V^{n}(u_{n})| < \frac{\kappa}{2}.
      \end{equation}
      Then the definition of $V^{n}(u_{n})$ in \eqref{eq:Vny} ensures that for any $\kappa > 0$ there exist $T_{n,\kappa} > 0, z_{n,\kappa} \in C([0,T_{n,\kappa}];H_{n}), z_{n,\kappa}(0) = 0$ and  $z_{n,\kappa}(T_{n,\kappa}) = u_{n}$ such that
      \begin{equation}\label{eq:518}
            I_{0,T_{n,\kappa}}^{n,0}(z_{n,\kappa})
            \leq
            V^{n}(u_{n}) + \frac{\kappa}{4}
            <
            +\infty,
      \end{equation}
      which implies $z_{n,\kappa} \in \mathcal{H}_{1}^{n,0}(T_{n,\kappa})$ by \eqref{eq:I0Tnyz}. According to \eqref{eq:Vntauy}, we have
      \begin{equation}\label{eq:519}
            V^{n,\tau}(u_{n})
            \leq
            I_{0,T_{n,\kappa}}^{n,\tau,0}(z_{n,\kappa}).
      \end{equation}
      As \eqref{eq:410410410} in Theorem \ref{th:fullLDPasympre} ensures
      $
            \lim\limits_{\tau \to 0}I_{0,T_{n,\kappa}}^{n,\tau,0}(z_{n,\kappa})
            =
            I_{0,T_{n,\kappa}}^{n,0}(z_{n,\kappa})$, i.e.,
      there exists $\tau > 0$ such that
      \begin{equation}\label{eq:520}
            \big|I_{0,T_{n,\kappa}}^{n,\tau,0}(z_{n,\kappa})
            -
            I_{0,T_{n,\kappa}}^{n,0}(z_{n,\kappa})\big|
            \leq
            \frac{\kappa}{4},
      \end{equation}
      it follows from \eqref{eq:518}, \eqref{eq:519} and \eqref{eq:520} that
      \begin{equation}\label{eq:521}
            V^{n,\tau}(u_{n})
            \leq
            I_{0,T_{n,\kappa}}^{n,\tau,0}(z_{n,\kappa})
            \leq
            I_{0,T_{n,\kappa}}^{n,0}(z_{n,\kappa})
            +
            \frac{\kappa}{4}
            \leq
            V^{n}(u_{n}) + \frac{\kappa}{2}.
      \end{equation}
      It remains to show $V^{n}(u_{n}) \leq V^{n,\tau}(u_{n}) + \frac{\kappa}{2}$ for the above given $\tau > 0$. The definition of $V^{n,\tau}(u_{n})$ in \eqref{eq:Vntauy} implies that for $\kappa > 0$ there exists $T_{n,\tau,\kappa} > 0, z_{n,\tau,\kappa} \in C([0,T_{n,\tau,\kappa}];H_{n}), z_{n,\tau,\kappa}(0) = 0$ and $z_{n,\tau,\kappa}(T_{n,\tau,\kappa}) = u_{n}$ such that
      \begin{equation}\label{eq:522}
            I_{0,T_{n,\tau,\kappa}}^{n,\tau,0}(z_{n,\tau,\kappa})
            \leq
            V^{n,\tau}(u_{n}) + \frac{\kappa}{2}
            \leq
            V^{n}(u_{n}) + \kappa
            <
            +\infty,
      \end{equation}
      which together with \eqref{eq:I0TnMyz} implies $z_{n,\tau,\kappa} \in \mathcal{H}_{1}^{n,0}(T_{n,\tau,\kappa})$.
      The definition of $V^{n}(u_{n})$ in \eqref{eq:Vny} yields
      \begin{equation}\label{eq:523}
            V^{n}(u_{n})
            \leq
            I_{0,T_{n,\tau,\kappa}}^{n,0}(z_{n,\tau,\kappa}).
      \end{equation}
      Taking advantaging of $F \equiv \textbf{0}$, we have
      \begin{equation}\label{eq:524}
            I_{0,T_{n,\tau,\kappa}}^{n,0}(z_{n,\tau,\kappa})
            =
            \frac{1}{2}
            \int_{0}^{T_{n,\tau,\kappa}}
            \Big|Q_{n}^{-\frac12}\Big(\frac{\diff{z_{n,\tau,\kappa}(t)}}{\diff{t}}
                 -
                 A_{n}z_{n,\tau,\kappa}(t)
                 \Big)
            \Big|^{2}
            \diff{t}
            =
            I_{0,T_{n,\tau,\kappa}}^{n,\tau,0}(z_{n,\tau,\kappa}).
      \end{equation}
      This together with \eqref{eq:522}, \eqref{eq:523} and \eqref{eq:524} leads to
      \begin{equation}\label{eq:525}
            V^{n}(u_{n})
            \leq
            I_{0,T_{n,\tau,\kappa}}^{n,0}(z_{n,\tau,\kappa})
            =
            I_{0,T_{n,\tau,\kappa}}^{n,\tau,0}(z_{n,\tau,\kappa})
            \leq
            V^{n,\tau}(u_{n}) + \frac{\kappa}{2}.
      \end{equation}
      Then \eqref{eq:521} and \eqref{eq:525} show
      \begin{equation}\label{eq:invapart222}
            |V^{n}(u_{n})-V^{n,\tau}(u_{n})| \leq \frac{\kappa}{2}.
      \end{equation}
      For $n,\tau$ given by \eqref{eq:invapart111} and \eqref{eq:520}, we set $u_{n,\tau} := u_{n}$. By \eqref{eq:invapart111}, \eqref{eq:invapart222} and the triangle inequality, we complete the proof.
\end{proof}

\begin{Remark}
       When $F \colon H \to H$ is a linear mapping, i.e., $F(u) = Bu$ for some $B \in \mathcal{L}(H)$ with $|B|_{\mathcal{L}(H)} < \lambda_{1}$, we can redefine the unbounded linear operator $A$ by $A+B$ to vanish $F$, 
       as well as the numerical discretizations. Repeating procedure in Theorems \ref{th:semiinvameaLDPasympre} and \ref{fullinvameaLDPasypre} still leads to the weakly asymptotical preservation for the corresponding LDPs by the numerical approximations.
\end{Remark}

\bibliographystyle{abbrv}

\begin{thebibliography}{}

\end{thebibliography}


\begin{thebibliography}{10}

\bibitem{brzezniak2017large}
Z.~Brze\'{z}niak and S.~Cerrai.
\newblock Large deviations principle for the invariant measures of the {2D}
  stochastic {Navier--Stokes} equations on a torus.
\newblock {\em J. Funct. Anal.}, 273(6):1891--1930, 2017.

\bibitem{cerrai2004large}
S.~Cerrai and M.~R\"{o}ckner.
\newblock Large deviations for stochastic reaction-diffusion systems with
  multiplicative noise and non-{L}ipschitz reaction term.
\newblock {\em Ann. Probab.}, 32(1B):1100--1139, 2004.

\bibitem{cerrai2005largeinvariant}
S.~Cerrai and M.~R\"{o}ckner.
\newblock Large deviations for invariant measures of stochastic
  reaction--diffusion systems with multiplicative noise and non-{L}ipschitz
  reaction term.
\newblock {\em Ann. I. H. Poincar\'{e}-PR}, 41(1):69--105, 2005.

\bibitem{chen2020symplectic}
C.~Chen.
\newblock A symplectic discontinuous {G}alerkin full discretization for
  stochastic {M}axwell equations.
\newblock {\em arXiv:2009.09880}.

\bibitem{chen2020large}
C.~Chen, J.~Hong, D.~Jin, and L.~Sun.
\newblock Large deviations principles for symplectic discretizations of
  stochastic linear {S}chr\"{o}dinger equation.
\newblock {\em arXiv:2006.01357}.

\bibitem{chen2021asymptotically}
C.~Chen, J.~Hong, D.~Jin, and L.~Sun.
\newblock Asymptotically-preserving large deviations principles by stochastic
  symplectic methods for a linear stochastic oscillator.
\newblock {\em SIAM J. Numer. Anal.}, 59(1):32--59, 2021.

\bibitem{chen2017approximation}
C.~Chen, J.~Hong, and X.~Wang.
\newblock Approximation of invariant measure for damped stochastic nonlinear
  {S}chr{\"o}dinger equation via an ergodic numerical scheme.
\newblock {\em Potential Anal.}, 46(2):323--367, 2017.

\bibitem{CGW2020ErgodicityANM}
Z.~Chen, S.~Gan, and X.~Wang.
\newblock A full-discrete exponential {E}uler approximation of the invariant
  measure for parabolic stochastic partial differential equations.
\newblock {\em Appl. Numer. Math.}, 157:135--158, 2020.

\bibitem{cui2021weak}
J.~Cui, J.~Hong, and L.~Sun.
\newblock Weak convergence and invariant measure of a full discretization for
  parabolic {SPDEs} with non-globally {Lipschitz} coefficients.
\newblock {\em Stochastic Process. Appl.}, 134:55--93, 2021.

\bibitem{da2006introduction}
G.~Da~Prato.
\newblock {\em An Introduction to Infinite-Dimensional Analysis}.
\newblock Universitext. Springer Science and Business Media, Berlin, 2006.

\bibitem{da2014stochastic}
G.~Da~Prato and J.~Zabczyk.
\newblock {\em Stochastic Equations in Infinite Dimensions}.
\newblock Encyclopedia of Mathematics and Its Applications 152. Cambridge
  University Press, Cambridge, 2014.

\bibitem{dembo2009large}
A.~Dembo and O.~Zeitouni.
\newblock {\em Large Deviations Techniques and Applications}.
\newblock Stochastic Modelling and Applied Probability 38. Springer, Berlin,
  2009.

\bibitem{freidlin1988random}
M.~I. Freidlin.
\newblock Random perturbations of reaction-diffusion equations: the
  quasi-deterministic approximation.
\newblock {\em Trans. Amer. Math. Soc.}, 305(2):665--697, 1988.

\bibitem{freidlin1984random}
M.~I. Freidlin and A.~D. Wentzell.
\newblock {\em Random Perturbations of Dynamical Systems}.
\newblock Grundlehren der mathematischen Wissenschaften 260. Springer, New
  York, 1984.

\bibitem{gadat2013large}
S.~Gadat, F.~Panloup, and C.~Pellegrini.
\newblock Large deviation principle for invariant distributions of memory
  gradient diffusions.
\newblock {\em Electron. J. Probab.}, 18(81):1--34, 2013.

\bibitem{hong2021numerical}
J.~Hong, D.~Jin, and D.~Sheng.
\newblock Numerical approximations of one-point large deviations rate functions
  of stochastic differential equations with small noise.
\newblock {\em arXiv:2102.04061}.

\bibitem{hong2020numerically}
J.~Hong, D.~Jin, D.~Sheng, and L.~Sun.
\newblock Numerically asymptotical preservation of the large deviations
  principles for invariant measures of {L}angevin equations.
\newblock {\em arXiv:2009.13336}.

\bibitem{hong2019invariant}
J.~Hong and X.~Wang.
\newblock {\em Invariant Measures for Stochastic Nonlinear Schr\"{o}dinger
  Equations: Numerical Approximations and Symplectic Structures}.
\newblock Lecture Notes in Mathematics 2251. Springer, Singapore, 2019.

\bibitem{hong2017numerical}
J.~Hong, X.~Wang, and L.~Zhang.
\newblock Numerical analysis on ergodic limit of approximations for stochastic
  {NLS} equation via multi-symplectic scheme.
\newblock {\em SIAM J. Numer. Anal.}, 55(1):305--327, 2017.

\bibitem{jentzen2009overcoming}
A.~Jentzen and P.~E. Kloeden.
\newblock Overcoming the order barrier in the numerical approximation of
  stochastic partial differential equations with additive space-time noise.
\newblock {\em Proc. R. Soc. A}, 465(2102):649--667, 2009.

\bibitem{kloeden1992numerical}
P.~E. Kloeden and E.~Platen.
\newblock {\em Numerical Solution of Stochastic Differential Equations}.
\newblock Application of Mathematics, Stochastic Modelling and Applied
  Probability 23. Springer, Berlin, 1992.

\bibitem{kruse2014strong}
R.~Kruse.
\newblock {\em Strong and Weak Approximation of Semilinear Stochastic Evolution
  Equations}.
\newblock Lecture Notes in Mathematics 2093. Springer, New York, 2014.

\bibitem{lord2014introduction}
G.~J. Lord, C.~E. Powell, and T.~Shardlow.
\newblock {\em An Introduction to Computational Stochastic {PDE}s}.
\newblock Number~50 in Cambridge Texts in Applied Mathematics. Cambridge
  University Press, 2014.

\bibitem{peszat1994large}
S.~Peszat.
\newblock Large deviation principle for stochastic evolution equations.
\newblock {\em Probab. Theory Relat. Fields}, 98(1):113--136, 1994.

\bibitem{prevot2007concise}
C.~Pr{\'e}v{\^o}t and M.~R{\"o}ckner.
\newblock {\em A Concise Course on Stochastic Partial Differential Equations}.
\newblock Lecture Notes in Mathematics 1905. Springer, New York, 2007.

\bibitem{sowers1992largeinvariant}
R.~Sowers.
\newblock Large deviations for the invariant measure of a reaction-diffusion
  equation with non-{G}aussian perturbations.
\newblock {\em Probab. Theory Relat. Fields}, 92(3):393--421, 1992.

\bibitem{sowers1992large}
R.~B. Sowers.
\newblock Large deviations for a reaction-diffusion equation with
  non-{G}aussian perturbations.
\newblock {\em Ann. Probab.}, 20(1):504--537, 1992.

\bibitem{wang2015note}
X.~Wang and R.~Qi.
\newblock A note on an accelerated exponential {E}uler method for parabolic
  {SPDE}s with additive noise.
\newblock {\em Appl. Math. Lett.}, 46:31--37, 2015.

\end{thebibliography}

\end{document}